\documentclass[a4paper,11pt]{article}
\usepackage{amsfonts}
\usepackage{pifont}
\usepackage[title]{appendix}
\usepackage{caption}
\usepackage{graphicx, subfig}
\usepackage{bm}
\usepackage{latexsym,amsmath,amssymb,cite,amsthm}
\usepackage{color,eucal,enumerate,mathrsfs}
\usepackage{microtype}
\usepackage{tikz}
\DeclareFontFamily{TS1}{cmr}{}
\DeclareFontShape{TS1}{cmr}{m}{n}{<->ssub*lmr/m/n}{}
\usepackage[pagewise]{lineno}
\usepackage[pagebackref=true, colorlinks, linkcolor=blue, anchorcolor=red,citecolor=blue,urlcolor=blue]{hyperref}
\hypersetup{
  pdftitle={Two-mode stability for multi-marginal optimal transport maps},
  pdfauthor={Bang-Xian Han and Zhuo-Nan Zhu},
  pdfsubject={Quantitative stability in multi-marginal optimal transport},
  pdfkeywords={multi-marginal optimal transport, quantitative stability, Wasserstein barycenter}
}
\allowdisplaybreaks

\numberwithin{equation}{section}
\theoremstyle{plain}
\newtheorem{theorem}{Theorem}[section]
\newtheorem{corollary}[theorem]{Corollary}
\newtheorem{lemma}[theorem]{Lemma}
\newtheorem{proposition}[theorem]{Proposition}
\theoremstyle{definition}

\newtheorem{example}[theorem]{Example}

\theoremstyle{remark}

\newtheorem{remark}[theorem]{Remark}
\newtheorem*{proposition*}{Proposition}

\newcommand{\spt}{\operatorname{supp}}
\newcommand{\Lip}{\operatorname{Lip}}
\newcommand{\diam}{\operatorname{diam}}
\newcommand{\N}{\mathbb{N}}
\newcommand{\R}{\mathbb{R}}
\newcommand{\Id}{\operatorname{Id}}
\newcommand{\dist}{\operatorname{dist}}
\renewcommand{\d}{{\mathrm d}}
\newcommand{\restr}[1]{\lower3pt\hbox{$|_{#1}$}}

\title{\textbf{Two-mode stability for multi-marginal optimal transport maps}}
	\author{Bang-Xian Han\thanks{School of Mathematics, Shandong University, Jinan, China. Email: hanbx@sdu.edu.cn }
		\and Zhuo-Nan Zhu
		\thanks{School of Mathematical Sciences, University of Science and Technology of China, Hefei, China. Email: zhuonanzhu@mail.ustc.edu.cn}
	}
\date{October 19, 2026}
\begin{document}

		\maketitle

\begin{abstract}
We establish a two-mode stability theory for Monge solutions of the multi-marginal optimal transport problem with barycentric quadratic cost. The associated tuple of maps splits into an external barycentric mode and internal relative modes. A quadratic lower bound for the Kantorovich defect controls the internal modes and yields a square-root estimate without invoking any two-marginal stability result. The external mode is the optimal transport map from the fixed source to the Wasserstein barycenter. Combining the resulting two-mode estimate with M\'erigot's sharp stability theorem gives a $\frac{1}{4}$-H\"older estimate for general target perturbations when the source has bounded density, while barycenter-preserving perturbations satisfy a $\frac{1}{2}$-H\"older estimate. We prove that both exponents and the dependence on the target weights are optimal in dimensions $n\geq2$. We also extend the two-mode estimate to collective-coordinate perturbations and uniformly concave costs of the sum. For nonnegative pairwise quadratic interactions satisfying a spectral condition, we prove stability under perturbations of the target marginals and the cost, allowing approximately optimal comparison plans.
\end{abstract}



		\tableofcontents

\section{Introduction}\label{sec:introduction}

We study the stability of multi-marginal optimal transport maps when the target marginals vary and the source is fixed. For the barycentric quadratic cost, an exact identity for the Kantorovich defect separates the motion of the barycenter from the relative motion of the target coordinates. This identity yields the exponent $\frac12$ for the internal modes, while general perturbations inherit the sharp exponent $\frac14$ from the two-marginal stability of the barycenter map. The same separation also leads to stability results for collective-coordinate costs and a class of pairwise quadratic interactions.

\medskip

\paragraph{Multi-marginal optimal transport.}
Fix $N\ge 3$ and weights
\begin{equation*}
    \lambda_i>0,
    \qquad
    \sum_{i=1}^N \lambda_i=1.
\end{equation*}
Let $S,\, Y\subseteq \mathbb R^n$ be compact, let $\rho\in \mathcal P(S)$, and let $\mu_2,\dots,\mu_N\in\mathcal P(Y)$. Throughout, coordinate $1$ is the fixed source: we set $\mu_1=\tilde\mu_1=\rho$ and $I:=\{2,\dots,N\}$ for the target indices. In the main part of this paper we study the cost
\begin{equation}\label{cost}
    c(x_1,\dots,x_N):=\inf_{y\in\mathbb R^n}\sum_{i=1}^N \frac{\lambda_i}{2}|x_i-y|^2.
\end{equation}
This weighted quadratic cost appears in the multi-marginal theory of Gangbo and \'Swi\c{e}ch \cite{zbMATH01060715} and in the study of the Wasserstein barycenter problem. We call it the barycentric quadratic cost.

For the barycentric quadratic cost \eqref{cost}, the Monge problem \cite{monge1781} is to solve
\begin{equation}\label{eq:monge-problem}
    \inf_{T_2,\dots,T_N}\int_S c(x_1,T_2(x_1),\dots,T_N(x_1))\,\d\rho(x_1)
\end{equation}
among maps $T_i:S\to Y$ satisfying $(T_i)_\#\rho=\mu_i$, $i=2,\dots,N$. Its Kantorovich relaxation \cite{zbMATH03099866} is
\begin{equation}\label{eq:kantorovich-relaxation}
    \inf_{\pi\in \Pi(\rho,\mu_2,\dots,\mu_N)}\int_{S\times Y^{N-1}}c(x_1,\dots,x_N)\,\d\pi(x_1,\dots,x_N),
\end{equation}
where $\Pi(\rho,\mu_2,\dots,\mu_N)$ is the set of probability measures with the prescribed marginals. The minimizer in \eqref{cost} is the Euclidean weighted barycenter $\sum_{i=1}^N\lambda_i x_i$, and hence
\begin{equation}\label{equal}
    c(x_1,\dots,x_N)=\frac12\sum_{i=1}^N\lambda_i|x_i|^2-
    \frac12\left|\sum_{i=1}^N\lambda_i x_i\right|^2
    =\frac12\sum_{1\le i<j\le N}\lambda_i\lambda_j|x_i-x_j|^2.
\end{equation}

Gangbo and \'Swi\c{e}ch's theorem \cite{zbMATH01060715} (see also Kim and Pass \cite{zbMATH06484002}) extends Brenier's theorem \cite{zbMATH00011193} to this multi-marginal setting: if $\rho$ is absolutely continuous, then the Monge problem has a unique solution $(T_2,\dots,T_N)$. We call these maps multi-marginal optimal transport maps.

Multi-marginal optimal transport arises in the study of Wasserstein barycenters, economic matching and density functional theory. We refer to \cite{zbMATH05956494,chiappori2010hedonic,pass2014matching,zbMATH07098220,arXiv:2509.22494} for these applications and further developments, and to the survey \cite{pass2015multi} for a broader account.

\medskip

\paragraph{The two-mode decomposition and Wasserstein barycenters.}
Let
\begin{equation*}
    \Omega:=\overline{\operatorname{conv}(S\cup Y)}\subseteq\mathbb R^n,
    \quad
    \mathbb P:=\lambda_1\delta_\rho+\sum_{i=2}^N\lambda_i\delta_{\mu_i}\in \mathcal P(\mathcal P(\Omega)).
\end{equation*}
The Wasserstein barycenter of $\mathbb P$ is a minimizer of
\begin{equation*}
    \mathcal{F}_{\mathbb P}(\nu):=\frac{\lambda_1}{2}W_2^2(\nu,\rho)+\sum_{i=2}^N\frac{\lambda_i}{2}W_2^2(\nu,\mu_i),
    \quad \nu\in\mathcal P_2(\mathbb R^n).
\end{equation*}
The following proposition links the Wasserstein barycenter to the barycentric multi-marginal optimal transport problem (see \cite{zbMATH01060715}, \cite[Section 4]{zbMATH05956494}, and \cite[Proposition 5.1]{zbMATH06484002}).

\begin{proposition}[Barycenter map]\label{prop:barycentermap}
Let $\mathbb P=\lambda_1\delta_\rho+\sum_{i=2}^{N}\lambda_i\delta_{\mu_i}\in \mathcal P(\mathcal P(\Omega))$ and assume $\rho\ll \mathcal L^n$. Then there exists a unique Wasserstein barycenter of $\mathbb P$, denoted by $\mu_{\mathbb P}\in\mathcal P(\Omega)$, and $\mu_{\mathbb P}\ll\mathcal L^n$. Moreover, if $(T_2,\dots,T_N)$ are the multi-marginal optimal transport maps, then
\begin{equation*}
    B:=\lambda_1\Id+\sum_{i=2}^N\lambda_iT_i
\end{equation*}
is the unique optimal transport map from $\rho$ to $\mu_{\mathbb P}$, for the quadratic cost. We call $B$ the barycenter map.
\end{proposition}

This characterization gives a two-mode decomposition of the transport maps. The barycenter map $B$ describes the averaged motion and forms the external mode. Each individual map also contains a relative component:
\[
    V_i:=T_i-M,\qquad
    M:=\frac{B-\lambda_1\Id}{1-\lambda_1},
    \qquad
    \sum_{i=2}^N\lambda_iV_i=0.
\]
The variables $(V_i)_{i\geq2}$ form the internal modes: they describe the relative positions of the target points inside a fiber of the barycenter map. The decomposition is already present in the cost itself:
\begin{equation}\label{eq:intro-cost-splitting}
 c(x_1,\dots,x_N)
 =\frac{\lambda_1}{2(1-\lambda_1)}|B_x-x_1|^2
   +\frac12\sum_{i=2}^N\lambda_i|V_{i,x}|^2,
\end{equation}
where $B_x=\sum_i\lambda_i x_i$, $M_x=(B_x-\lambda_1x_1)/(1-\lambda_1)$ and $V_{i,x}=x_i-M_x$.

\medskip

\paragraph{Main results.}
We write $\mathcal P_{\mathrm{ac}}(S)$ for the set of probability measures on $S$ that are absolutely continuous with respect to the Lebesgue measure $\mathcal L^n$. Given another tuple of target marginals $\tilde\mu_2,\dots,\tilde\mu_N\in\mathcal P(Y)$, let $(\tilde T_2,\dots,\tilde T_N)$ be the corresponding multi-marginal optimal transport maps and set
\begin{equation*}
    \tilde B:=\lambda_1\Id+\sum_{i=2}^N\lambda_i\tilde T_i,
    \quad
    \tilde{\mathbb P}:=\lambda_1\delta_\rho+\sum_{i=2}^N\lambda_i\delta_{\tilde\mu_i}.
\end{equation*}

To obtain a stability estimate from this separation, we work with the Kantorovich defect $R=c-\sum_i\varphi_i\geq0$. For a fixed contact tuple $(x_1,T_2(x_1),\dots,T_N(x_1))$, write $\Delta x_i=x_i-T_i(x_1)$ and let $r_i\geq0$ denote the defect obtained by replacing its $i$th coordinate with $x_i$. The full defect and these single-coordinate defects satisfy the exact identity
\begin{equation}\label{eq:intro-defect-identity}
 2R(x_1,\dots,x_N)+\left|\sum_{i=2}^N\lambda_i\Delta x_i\right|^2
 =2\sum_{i=2}^N r_i+\sum_{i=2}^N\lambda_i^2|\Delta x_i|^2.
\end{equation}
This identity shows that a small defect controls every target coordinate once the collective displacement is controlled, and hence gives the internal estimate without invoking a two-marginal stability theorem. To control the remaining external mode, we identify it with the optimal transport map to the Wasserstein barycenter and apply M\'erigot's estimate.

The two stability exponents therefore have different origins: general perturbations inherit the $\frac{1}{4}$ exponent from the external barycenter mode, whereas barycenter-preserving perturbations and self-comparisons eliminate this contribution and retain the $\frac{1}{2}$ exponent coming from internal quadratic coercivity. For pairwise self-comparisons, we also determine how the constants depend on small target weights and prove that this dependence is optimal for each stability exponent under consideration.

\medskip

Our first result gives a two-mode stability estimate for the maps without using a two-marginal stability theorem.
\begin{theorem}\label{main}
Let $N\geq 3$ and let $S,\, Y\subseteq \R^n$ be compact. Let $\rho\in\mathcal P_{\mathrm{ac}}(S)$. Then there exists a constant $C>0$, depending on $N$ and $\diam(S\cup Y)$, such that for any $\{\mu_i\}_{i=2}^N, \{\tilde{\mu}_i\}_{i=2}^N\subseteq \mathcal{P}(Y)$,
\begin{equation*}
    \sum_{i=2}^N\lambda_i\|T_i-\tilde T_i\|_{L^2(\rho)}
    \leq C\left(\left(\sum_{i=2}^N \lambda_iW_1(\mu_i,\tilde\mu_i)\right)^{1/2}+\|B-\tilde B\|_{L^2(\rho)}\right).
\end{equation*}
Here $(T_2,\dots,T_N)$ and $(\tilde T_2,\dots,\tilde T_N)$ are the multi-marginal optimal transport maps associated with $(\rho,\mu_2,\dots,\mu_N)$ and $(\rho,\tilde\mu_2,\dots,\tilde\mu_N)$, respectively.
\end{theorem}

The first term in Theorem \ref{main} corresponds to the \emph{internal multi-marginal mode}. The second term is the \emph{external mode}: by Proposition \ref{prop:barycentermap}, it is the distance between two barycenter maps. Combining this theorem with M\'erigot's sharp two-marginal stability estimate \cite{merigot:hal-05616391} yields the following general perturbation estimate.
\begin{theorem}\label{multi-marginal}
Let $N\geq 3$ and let $S,\,Y\subseteq\R^n$ be compact. Let $\rho\in\mathcal P_{\mathrm{ac}}(S)$ satisfy $\left\|\frac{\d\rho}{\d\mathcal L^n}\right\|_{L^\infty}<\infty$. With the notation of Theorem \ref{main}, there exists a constant $C>0$, depending on $\lambda_1$, $N$, $\diam(S\cup Y)$, $n$ and $\|\rho\|_{L^\infty}$, such that
\begin{equation*}
    \sum_{i=2}^N\lambda_i\|T_i-\tilde T_i\|_{L^2(\rho)}
    \leq C\left(\sum_{i=2}^N \lambda_iW_1(\mu_i,\tilde\mu_i)\right)^{1/4}.
\end{equation*}
Moreover, the exponent $\frac{1}{4}$ is optimal in dimension $n\ge 2$.
\end{theorem}

Here and below, $\|\rho\|_{L^\infty}$ denotes $\left\|\frac{\d\rho}{\d\mathcal L^n}\right\|_{L^\infty}$. When using function-space notation, we identify absolutely continuous measures with their densities.

\begin{remark}\label{rem:one-dim}
The dimensional restriction $n\ge 2$ in the optimality statement is essential. In dimension $n=1$, the barycentric multi-marginal optimal transport maps are given by monotone rearrangements on the line (cf. \cite[Chapter 1]{zbMATH05306371}). Hence
\begin{equation*}
    \|T_i-\tilde T_i\|_{L^2(\rho)}=W_2(\mu_i,\tilde\mu_i)
    \leq \diam(Y)^{1/2}W_1^{1/2}(\mu_i,\tilde\mu_i),
    \quad i=2,\dots,N.
\end{equation*}
Therefore the general perturbation estimate holds with exponent $\frac{1}{2}$ in dimension $n=1$. This exponent is already optimal in dimension one, as the atomic example in Proposition~\ref{exponent}\textup{(a)} shows.
\end{remark}

When $B=\tilde B$, the external contribution in Theorem~\ref{main} vanishes, leaving only the internal coercive estimate. This gives the stronger exponent $\frac12$ without an $L^\infty$ bound on the source density, as stated in the following corollary.
\begin{corollary}\label{cor:barycenter-preserving-stability}
Under the same assumptions and notation as in Theorem \ref{main}, assume moreover that
$B=\tilde B$.
Then there exists a constant $C>0$, depending only on $N$ and $\diam(S\cup Y)$, such that
\begin{equation*}
    \sum_{i=2}^N\lambda_i\|T_i-\tilde T_i\|_{L^2(\rho)}
    \leq C\left(\sum_{i=2}^N\lambda_i W_1(\mu_i,\tilde\mu_i)\right)^{1/2}.
\end{equation*}
\end{corollary}

\begin{remark}
The exponent $\frac{1}{2}$ in Corollary \ref{cor:barycenter-preserving-stability} is reminiscent of Ambrosio-type $\frac{1}{2}$ stability estimates for two-marginal optimal transport maps reported in \cite{zbMATH05907251}. Here the source of the estimate is different: rather than relying on regularity of the optimal transport map, the estimate follows from the internal quadratic coercivity of the barycentric quadratic cost once the external barycenter mode has been removed.
\end{remark}

\medskip

We say that $\{\tilde{\mu}_i\}_{i=2}^N$ is a \emph{barycenter-preserving perturbation} of $\{\mu_i\}_{i=2}^N$
if $\mu_{\mathbb P}=\mu_{\tilde{\mathbb P}}$.
This condition is equivalent, by uniqueness of the optimal transport map from $\rho$, to $B=\tilde B$. In particular, any weight-preserving permutation of the target marginals preserves the barycenter map. This yields the following self-comparison result. Barycenter-preserving perturbations are not limited to permutations of the marginal measures; in \S\ref{sec:barycenter-preserving}, we give two non-permutation examples.

\begin{theorem}\label{self-comparison}
Under the same assumptions and notation as in Theorem \ref{main}, let $\sigma$ be a permutation of $\{2,\dots,N\}$ such that
\begin{equation*}
    \lambda_{\sigma(i)}=\lambda_i,
    \quad i=2,\dots,N.
\end{equation*}
Then there exists a constant $C>0$, depending on $N$ and $\diam(S\cup Y)$, such that
\begin{equation*}
    \sum_{i=2}^N\lambda_i\|T_i-T_{\sigma(i)}\|_{L^2(\rho)}
    \leq C\left(\sum_{i=2}^N\lambda_i W_1(\mu_i,\mu_{\sigma(i)})\right)^{1/2}.
\end{equation*}
In particular, if $\lambda_i=\lambda_j$, then
\begin{equation*}
    \|T_i-T_j\|_{L^2(\rho)}
    \leq C(\lambda_i+\lambda_j)^{-1/2}W_1^{1/2}(\mu_i,\mu_j).
\end{equation*}
Moreover, the exponent $\frac12$ is optimal in dimension $n\ge 1$. In dimension $n\ge 2$, if one keeps the exponent $\frac12$ in $W_1(\mu_i,\mu_j)$, then the power $-\frac12$ of $(\lambda_i+\lambda_j)$ is optimal as well.
\end{theorem}

Combining Theorem \ref{self-comparison} with M\'erigot's sharp two-marginal stability estimate gives two complementary bounds. The first does not contain the factor $(\lambda_i+\lambda_j)$ and has exponent $\frac{1}{4}$; the second has the stronger exponent $\frac{1}{2}$ and necessarily includes the weight factor.

\begin{corollary}\label{cor:pairwise-two-scale}
Under the same assumptions and notation as in Theorem \ref{multi-marginal}, assume moreover that $\lambda_i=\lambda_j$. Then there exists a constant $C>0$, depending on $\lambda_1$, $N$, $\diam(S\cup Y)$, $n$ and $\|\rho\|_{L^\infty}$, such that
\begin{equation*}
    \|T_i-T_j\|_{L^2(\rho)}
    \leq C\min\left\{W_1^{1/4}(\mu_i,\mu_j),\, (\lambda_i+\lambda_j)^{-1/2}W_1^{1/2}(\mu_i,\mu_j)\right\}.
\end{equation*}
In particular, for any $\theta\in [\frac{1}{4}, \frac{1}{2}]$, one has
\begin{equation*}
      \|T_i-T_j\|_{L^2(\rho)}\leq C (\lambda_i+\lambda_j)^{1/2-2\theta}W_1^\theta(\mu_i,\mu_j).
\end{equation*}
Writing $D=\diam(S\cup Y)$, an admissible choice is
\[
 C=C_{N,n}D^{1/2}\left[1+\lambda_1^{-n/4}
       \bigl(D^n\|\rho\|_{L^\infty}\bigr)^{1/4}\right].
\]
In particular, the constant may be chosen uniformly for $\lambda_1\geq\lambda_*>0$. The target-weight exponents are optimal in this uniform sense, as specified in Proposition~\ref{exponent}.
\end{corollary}

For fixed positive weights, the second branch already gives exponent $\frac12$; the $\frac14$ branch becomes relevant when the constant must remain uniform as $\lambda_i+\lambda_j\downarrow0$. The following proposition describes precisely which exponent and weight dependence can hold simultaneously. Its proof is given in \S\ref{appendix:optimality}, using the construction in Appendix~\ref{app:fixed-source-construction}. In that construction, the target weight $\lambda_k=2^{-(k+k_0)^2}$ decays faster than every power of the cell size $s_k=\eta2^{-k}$. We choose the fixed integer $k_0$ so that $(N-1)\lambda_k\leq\frac12$ for all $k\geq1$, which keeps the source weight $\lambda_1^k=1-(N-1)\lambda_k$ bounded below.
\begin{proposition}\label{exponent}
The estimates in Theorem~\ref{self-comparison} and Corollary~\ref{cor:pairwise-two-scale} have the following optimality properties.
\begin{itemize}
 \item[\textup{\textbf{(a)}}] The exponent $\frac12$ in Theorem~\ref{self-comparison} is optimal in dimension $n\geq1$.
 \item[\textup{\textbf{(b)}}] For every $n\geq2$ and $N\geq3$, there exist fixed compact sets $S,Y\subset\R^n$ and a fixed bounded-density source $\rho\in\mathcal P_{\mathrm{ac}}(S)$ for which no constant $C$, uniform over the weights with $\lambda_1\geq\frac12$ and $\lambda_2=\lambda_3$, can make
 \[
 \|T_2-T_3\|_{L^2(\rho)}\leq C(\lambda_2+\lambda_3)^\gamma W_1^\theta(\mu_2,\mu_3)
 \]
 hold for all target tuples whenever $\theta\geq\frac14$ and $\gamma>\frac12-2\theta$.
\end{itemize}
In particular, on $\frac14\leq\theta\leq\frac12$, the weight exponent $\frac12-2\theta$ in Corollary~\ref{cor:pairwise-two-scale} cannot be increased. Setting $\gamma=0$ excludes every weight-uniform exponent $\theta>\frac14$; setting $\theta=\frac12$ gives the optimal weight exponent $-\frac12$.
\end{proposition}

\medskip

\paragraph{Relation to two-marginal stability.}
For two-marginal quadratic transport, quantitative stability under target perturbations was developed by Berman \cite{zbMATH07404754}, Delalande and M\'erigot \cite{zbMATH07794624}, and Letrouit and M\'erigot \cite{arXiv:2411.04908}. Cazelles, Pauwels and Portales \cite{cazelles2026statisticalestimationmongetransport} obtained the exponent $\frac15$ under bounded-density and compact-support assumptions; M\'erigot \cite{merigot:hal-05616391} proved the sharp exponent $\frac14$. See also \cite{letrouit2026unstable} for instability examples and \cite{LetrouitLecture} for the multi-marginal stability question. In our proof of the general perturbation estimate, the two-marginal theorem enters through the external barycenter map, whereas the internal estimate follows directly from \eqref{eq:intro-defect-identity}. Both parts of the argument naturally involve $W_1$, because Lipschitz dual potentials control changes in the dual objective. Since $W_1\leq W_p$ for $p\geq1$, our estimates also hold with $W_p$ and the same exponents, without asserting their optimality for those distances.

\medskip

The factorization $T_i=S_i\circ B$, where $S_i$ transports $\mu_{\mathbb P}$ optimally to $\mu_i$, might suggest applying two-marginal stability twice. Under a general perturbation, however, $\tilde T_i=\tilde S_i\circ\tilde B$ involves a different source measure for the outer map: $S_i$ is defined $\mu_{\mathbb P}$-a.e., while $\tilde S_i$ is defined $\mu_{\tilde{\mathbb P}}$-a.e. The cited fixed-source theorem therefore does not directly compare them, and splitting the difference of compositions requires both suitable extensions and control of their evaluation at different barycenter points. Additional estimates for this change of source would be needed to complete such an argument.

Even when the barycenter is unchanged, applying the bounded-density two-marginal theorem at the common barycenter gives only exponent $\frac14$, with a constant involving its density. The defect identity gives a different estimate: by comparing all coordinates over the original source $\rho$, it gives the square-root internal bound without requiring a bounded source density and retains the target weights explicitly. Thus the factorization explains the barycenter's role, while the internal identity provides the stronger control needed for barycenter-preserving perturbations and weight-dependent self-comparisons.

\medskip

\paragraph{Further costs and approximate plans.}
The exact identity also identifies the relevant condition for pairwise quadratic interactions
\[
 c_A(x_1,\dots,x_N)=\frac12\sum_{1\leq i<j\leq N}a_{ij}|x_i-x_j|^2,
 \qquad A=(a_{ij})=A^{\mathsf T},\quad a_{ii}=0,\quad a_{ij}\geq0.
\]
Pass and Vargas-Jim\'enez \cite{PassVargasJimenez2023} established Monge solutions and uniqueness for classes of pairwise interactions encoded by graphs. Our focus here is a quantitative defect estimate for arbitrary plans with a fixed source marginal, from which uniqueness follows when the defect vanishes.

With $I$ as above, let $a=(a_{12},\dots,a_{1N})^{\mathsf T}$ and $A_I=(a_{ij})_{i,j\in I}$. If $A$ has one positive and $N-1$ negative eigenvalues, then $-A_I$ is positive definite on $a^\perp$. These are exactly the relative directions that leave the source derivative unchanged; this derivative depends on the target coordinates through $\sum_{i=2}^Na_{1i}x_i$. Theorem~\ref{thm:quadratic-source-fixed} combines this relative coercivity with a convex-potential estimate to control all coordinates. In particular, for a bounded source density and any probability plan $\Gamma$ with first marginal $\rho$,
\[
 \left(\int\sum_{i=2}^N|x_i-T_i^A(x_1)|^2\,\d\Gamma\right)^{1/2}
 \leq C\left(\int R_A\,\d\Gamma\right)^{1/4},
\]
where $T_i^A$ are the reference optimal transport maps and $R_A$ is their Kantorovich defect. Since the other marginals of $\Gamma$ are arbitrary, Corollary~\ref{cor:quadratic-approximate} applies this estimate to perturbed target marginals and bounds the defect in terms of the marginal errors, a uniform cost perturbation and an optimality error. The comparison is made at the same source point throughout, so it controls the target coordinates without reassigning source points.

For sources satisfying $\rho(B_r(x))\leq M_\rho r^{n-1+\theta}$, $0<\theta\leq1$, the same argument gives the exponent $\theta/(2(1+\theta))$. Here the analytic input is the averaged subgradient-oscillation estimate of Ford~\cite[Theorem~2.2]{Ford2026}, which extends the convex regularity method used by M\'erigot. The multi-marginal part of the proof then combines control of the source derivative with coercivity of the interaction form to estimate every target coordinate. Alongside this result, Section~\ref{sec:beyond} establishes the two-mode estimates for collective-coordinate perturbations and uniformly concave costs of the sum.

\medskip

\paragraph{Organization of the paper.}
Section~\ref{sec:mode} develops the two-mode decomposition, proves the barycentric stability estimates, and establishes their optimality, including the dependence on small target weights. Section~\ref{sec:beyond} begins with examples of barycenter-preserving perturbations and then extends the method to collective-coordinate costs and pairwise quadratic interactions, including approximately optimal transport plans. The geometric construction used to prove optimality of the weight dependence is given in Appendix~\ref{app:fixed-source-construction}.

\section{Two-Mode Decomposition and Stability Estimates}\label{sec:mode}
\subsection{Barycentric and internal coordinates}\label{subsec:barycentric-internal-coordinates}
The proof of Theorem \ref{main} relies on separating the barycentric direction from the internal multi-marginal directions. We first introduce the corresponding barycentric coordinates.

		For $\{\mu_i\}_{i=2}^N, \{\tilde{\mu}_i\}_{i=2}^N\subseteq \mathcal{P}(Y)$ and $x_1\in S$, define the weighted means of the target maps
		\begin{align*}
				M(x_1)&:=\frac{B(x_1)-\lambda_1x_1}{1-\lambda_1}=\frac{1}{1-\lambda_1}\sum_{i=2}^N\lambda_iT_i(x_1),\\
				\tilde M(x_1)&:=\frac{\tilde B(x_1)-\lambda_1x_1}{1-\lambda_1}=\frac{1}{1-\lambda_1}\sum_{i=2}^N\lambda_i\tilde T_i(x_1).
			\end{align*}
		For $i=2,\dots,N$, define the internal directions
		\begin{equation*}
			V_i(x_1):=T_i(x_1)-M(x_1),\quad \tilde V_i(x_1):=\tilde T_i(x_1)-\tilde M(x_1), \quad x_1\in S.
		\end{equation*}
		Then we have
		\begin{equation*}
			\sum_{i=2}^N\lambda_iV_i(x_1)=0,\quad \sum_{i=2}^N\lambda_i\tilde V_i(x_1)=0, \quad x_1\in S.
		\end{equation*}

		For $x_1\in S$, $x_2,\dots, x_N\in Y$, define
		\begin{equation*}
			B_x:=\lambda_1x_1+\sum_{i=2}^N\lambda_i x_i,\quad M_x:=\frac{B_x-\lambda_1x_1}{1-\lambda_1}=\frac{1}{1-\lambda_1}\sum_{i=2}^N\lambda_i x_i;
		\end{equation*}
		and
		\begin{equation*}
		V_{i,x}:=x_i-M_x, \quad i=2,\dots,N, \quad \text{which implies}\quad \sum_{i=2}^N\lambda_iV_{i,x}=0.
		\end{equation*}
		Here $B_x,\, M_x$ and $V_{i,x}$ are the pointwise analogues of $B,\, M$ and $V_i$ associated with the tuple \((x_1,\dots,x_N)\).

\begin{remark}[Geometric interpretation of the decomposition]\label{rem:geometric-decomposition}
The decomposition $T_i=M+V_i$, with $\sum_{i=2}^N\lambda_iV_i=0$, separates the averaged coordinate from the relative coordinates, as in a center-of-mass decomposition. Although Lemma~\ref{lem:barycentric-decomposition} gives an exact splitting of the cost into squares, this alone does not imply coercivity of the Kantorovich defect, which also contains the dual potentials. The required estimate comes from Proposition~\ref{prop:kantorovich-defect-lower}: the nonnegative defects obtained by changing one coordinate of a contact tuple give \eqref{eq:intro-defect-identity}. When the collective displacement vanishes, the identity bounds the full defect below by $\frac12\sum_{i=2}^N\lambda_i^2|\Delta x_i|^2$, providing the internal coercivity used in the stability estimates.
\end{remark}

		\begin{lemma}\label{lem:barycentric-decomposition}
			For $x_1\in S$, $x_2,\dots, x_N\in Y$, we have
			\begin{equation}
				c(x_1,\dots,x_N)=\frac{\lambda_1}{2(1-\lambda_1)}|B_x-x_1|^2+\frac12\sum_{i=2}^N\lambda_i|V_{i,x}|^2.
			\end{equation}
		\end{lemma}
		\begin{proof}
			Using $V_{i,x}:=x_i-M_x$ for $ i=2,\dots,N$ together with \eqref{equal}, we obtain
			\begin{equation}
				\begin{aligned}
					c(x_1,\dots,x_N)&=\sum_{i=1}^N\frac{\lambda_i}{2} |x_i-B_x|^2\\
					&=\frac{\lambda_1}{2} |x_1-B_x|^2+\sum_{i=2}^N\frac{\lambda_i}{2} |V_{i,x}+M_x-B_x|^2\\
					&\overset{*}{=}\frac{\lambda_1}{2} |x_1-B_x|^2+\sum_{i=2}^N\frac{\lambda_i}{2} |M_x-B_x|^2+\sum_{i=2}^N\frac{\lambda_i}{2} |V_{i,x}|^2\\
					&=\frac{\lambda_1}{2(1-\lambda_1)}|B_x-x_1|^2+\frac12\sum_{i=2}^N\lambda_i|V_{i,x}|^2,
				\end{aligned}
			\end{equation}
			where $(*)$ follows from $\sum_{i=2}^N\lambda_iV_{i,x}=0$.
		\end{proof}

\subsection{Separation of barycentric and internal directions}\label{subsec:separation}
		Let $(\varphi_1,\dots,\varphi_N)$ be a $c$-conjugate tuple of Kantorovich potentials for the multi-marginal optimal transport problem $(\rho,\mu_2,\dots,\mu_N)$. For $x_1\in S$, $x_2,\dots, x_N\in Y$, define
		\begin{equation*}
			R(x_1,\dots,x_N):=c(x_1,\dots,x_N)-\sum_{i=1}^{N}\varphi_i(x_i).
		\end{equation*}
		Since $(\varphi_1,\dots,\varphi_N)$ is $c$-conjugate, we have
		\begin{equation*}
			R(x_1,\dots,x_N)\geq 0,\quad \forall\, (x_1,\dots,x_N)\in S\times Y^{N-1}.
		\end{equation*}
		Moreover, by the standard duality theorem for marginal problems (cf. \cite{zbMATH03850154}), in the form used for
		multi-marginal optimal transport (cf. \cite{zbMATH01060715}, \cite{zbMATH06484002}), we have
		\begin{equation}\label{eq:defect-zero}
			R(x_1,T_2(x_1),\dots,T_N(x_1))=0, \quad \rho\text{-a.e. }x_1\in S.
		\end{equation}

		The following proposition is the key estimate in the proof of Theorem \ref{main}.
		\begin{proposition}\label{prop:kantorovich-defect-lower}
For $\rho$-a.e. $x_1\in S$ and every $x_2,\dots,x_N\in Y$, set
\[
 \Delta x_i:=x_i-T_i(x_1),\qquad
 r_i:=R(x_1,T_2(x_1),\dots,T_{i-1}(x_1),x_i,T_{i+1}(x_1),\dots,T_N(x_1)).
\]
Then $r_i\geq0$ and
\begin{equation}\label{eq:exact-defect-identity}
 2R(x_1,\dots,x_N)+|B_x-B(x_1)|^2
 =2\sum_{i=2}^Nr_i+\sum_{i=2}^N\lambda_i^2|\Delta x_i|^2.
\end{equation}
In particular,
\begin{equation}\label{eq:kantorovich-defect-lower}
 R(x_1,x_2,\dots,x_N)\geq
 \frac14\sum_{i=2}^N\lambda_i^2|V_{i,x}-V_i(x_1)|^2
 -\frac32|B_x-B(x_1)|^2.
\end{equation}
\end{proposition}
\begin{proof}
Fix $x_1$ for which \eqref{eq:defect-zero} holds. When we subtract the sum of the one-coordinate defects from the full defect, the marginal potentials cancel, leaving only the mixed terms of the quadratic cost. Consequently,
\[
 R(x_1,\dots,x_N)
 =\sum_{i=2}^Nr_i-\sum_{2\leq i<j\leq N}\lambda_i\lambda_j
          \langle\Delta x_i,\Delta x_j\rangle.
\]
Since $\Delta B:=B_x-B(x_1)=\sum_{i=2}^N\lambda_i\Delta x_i$, expanding $|\Delta B|^2$ proves \eqref{eq:exact-defect-identity}. Dual feasibility gives $r_i\geq0$, so
\[
 \sum_{i=2}^N\lambda_i^2|\Delta x_i|^2\leq2R+|\Delta B|^2.
\]
Moreover, $\Delta V_i:=V_{i,x}-V_i(x_1)=\Delta x_i-\Delta B/(1-\lambda_1)$ and $\sum_{i=2}^N\lambda_i^2\leq(1-\lambda_1)^2$. Therefore
\[
 \sum_{i=2}^N\lambda_i^2|\Delta V_i|^2
 \leq2\sum_{i=2}^N\lambda_i^2|\Delta x_i|^2+2|\Delta B|^2
 \leq4R+4|\Delta B|^2.
\]
Rearranging the last inequality gives \eqref{eq:kantorovich-defect-lower}, in fact with a smaller coefficient in front of the collective term.
\end{proof}

\subsection{Gluing argument for perturbed marginals}\label{subsec:gluing}
		Denote
		\begin{equation*}
			D:=\diam(S\cup Y), 	\quad\delta:=\sum_{i=2}^N \lambda_i W_1(\mu_i,\tilde\mu_i)\leq D;
		\end{equation*}
		and denote
		\begin{equation*}
			\gamma:=(\Id,T_2,\dots,T_N)_\#\rho,\quad \tilde{\gamma}:=(\Id,\tilde T_2,\dots,\tilde T_N)_\#\rho.
		\end{equation*}

		\begin{lemma}\label{lem:gluing-coupling}
			There exists a coupling $\bar{\gamma}\in \Pi(\rho, \mu_2,\dots,\mu_N, \tilde{\mu}_2,\dots,\tilde{\mu}_N)$ with coordinates $(x_1, x_2,\dots, x_N, \tilde x_2,\dots,\tilde x_N)$ such that
			\begin{equation}\label{eq:gluing-marginals}
				(x_1,\tilde x_2,\dots,\tilde x_N)_\#\bar{\gamma}=\tilde{\gamma},\quad (x_i,\tilde{x}_i)_\#\bar{\gamma}=\pi_i,\quad i=2,\dots,N,
			\end{equation}
			where $\pi_i\in\Pi(\mu_i,\tilde{\mu}_i)$ is a $W_1$-optimal transport plan between $\mu_i$ and $\tilde{\mu}_i$.
		\end{lemma}
		\begin{proof}
We construct the coupling by disintegrating each $W_1$-optimal transport plan with respect to its second marginal:
\[
 \pi_i(\d x_i,\d\tilde x_i)
 =K_i(\tilde x_i,\d x_i)\,\tilde\mu_i(\d\tilde x_i).
\]
These conditional measures can then be combined with $\tilde\gamma$ along the $\tilde\mu_i$ marginals. More precisely, we set
\[
 \bar\gamma(\d x_1,\d x_2\cdots\d x_N,\d\tilde x_2\cdots\d\tilde x_N)
 =\tilde\gamma(\d x_1,\d\tilde x_2\cdots\d\tilde x_N)
   \prod_{i=2}^N K_i(\tilde x_i,\d x_i).
\]
Integrating out the other coordinates verifies both identities in \eqref{eq:gluing-marginals}. The construction is an explicit form of the iterated gluing lemma (cf.\ \cite[Lemma 2.1]{zbMATH08183830}).
		\end{proof}

		\begin{corollary}\label{cor:gluing-estimates}
				We have
			\begin{equation*}
				\left(\int |B_x-\tilde{B}(x_1)|^2\,\d\bar{\gamma}\right)^{1/2}\leq D^{1/2}\delta^{1/2}, \quad \left(\int 	\sum_{i=2}^N\lambda_i^2|V_{i,x}-\tilde V_i(x_1)|^2\,\d\bar\gamma\right)^{1/2}\leq D^{1/2}\delta^{1/2}.
			\end{equation*}
			Here $B_x$ and $V_{i,x}$ are associated with the coordinates $(x_1,x_2,\dots,x_N)$.
		\end{corollary}
		\begin{proof}
			Since
			\begin{equation*}
				(x_1,\tilde x_2,\dots,\tilde x_N)_\#\bar\gamma=\tilde\gamma=(\Id,\tilde T_2,\dots,\tilde T_N)_\#\rho,
			\end{equation*}
			we have
			\begin{equation*}
				\tilde x_i=\tilde T_i(x_1), \quad \bar\gamma\text{-a.e.}, \quad i=2,\dots,N.
			\end{equation*}
		\textbf{Estimate of the barycentric part:}
	    By the definitions of $B_x$ and $\tilde{B}(x_1)$, we have
		\begin{equation}
\begin{aligned}
    \left(\int |B_x-\tilde{B}(x_1)|^2\,\d\bar{\gamma}\right)^{1/2}
    &=\left(\int \left|\sum_{i=2}^{N}\lambda_i (x_i-\tilde{T}_i(x_1))\right|^2\,\d\bar{\gamma}\right)^{1/2}\\
    &\leq \sum_{i=2}^N\lambda_i\left(\int |x_i-\tilde x_i|^2\,\d\bar\gamma\right)^{1/2}\\
    &\overset{\eqref{eq:gluing-marginals}}{\leq} D^{1/2}\sum_{i=2}^N\lambda_i W_1^{1/2}(\mu_i,\tilde{\mu}_i)\\
&\le D^{1/2}\left(\sum_{i=2}^N\lambda_i W_1(\mu_i,\tilde\mu_i)\right)^{1/2}
=D^{1/2}\delta^{1/2}.
\end{aligned}
\end{equation}
\medskip
		\textbf{Estimate of the internal part:} By the definitions of $V_{i,x}$ and $\tilde V_i(x_1)$, we have
		\begin{equation}
			V_{i,x}-\tilde V_i(x_1)=(x_i-\tilde{T}_i(x_1))-\frac{1}{1-\lambda_1}\sum_{j=2}^{N}\lambda_j(x_j-\tilde{T}_j(x_1)).
		\end{equation}
		Using $\lambda_i<1$ for $i=2,\dots,N$, we have
		\begin{equation}
			\begin{aligned}
					\sum_{i=2}^N\lambda_i^2|V_{i,x}-\tilde V_i(x_1)|^2\leq&\sum_{i=2}^N\lambda_i|V_{i,x}-\tilde V_i(x_1)|^2\\
					 =&\sum_{i=2}^{N}\lambda_i|x_i-\tilde{T}_i(x_1)|^2-\frac{1}{1-\lambda_1}\left|\sum_{i=2}^{N}\lambda_i(x_i-\tilde{T}_i(x_1))\right|^2\\
					 \leq& \sum_{i=2}^{N}\lambda_i|x_i-\tilde{T}_i(x_1)|^2.
			\end{aligned}
		\end{equation}
		Integrating with respect to $\bar{\gamma}$ gives
		\begin{equation}
				\int\sum_{i=2}^N\lambda_i^2|V_{i,x}-\tilde V_i(x_1)|^2\,\d\bar\gamma\leq
			\sum_{i=2}^N\lambda_i\int |x_i-\tilde x_i|^2\,\d\bar\gamma\overset{\eqref{eq:gluing-marginals}}{\leq}D\sum_{i=2}^N\lambda_iW_1(\mu_i,\tilde\mu_i)=D\delta,
		\end{equation}
	     which proves the estimates.
		\end{proof}

			Denote
		\begin{equation*}
			\hat{\gamma}:=(x_1, x_2,\dots, x_N)_\#\bar{\gamma}\in \Pi(\rho,\mu_2, \dots,\mu_N).
		\end{equation*}

		\begin{lemma}\label{cost-difference}
			We have
			\begin{equation}\label{eq:cost-difference-bound}
				\left|\int c(x_1,\dots,x_N)\,\d\hat{\gamma}-\int c(x_1,\dots,x_N)\,\d\gamma\right|\leq 2D\delta.
			\end{equation}
		\end{lemma}
		\begin{proof}
			Recall that $\Omega=\overline{\operatorname{conv}(S\cup Y)}$. Since $\diam (\Omega)=D$, for any $x_1\in S, x_2,\dots,x_N\in Y$, we have
			\begin{equation}
				|\nabla_{x_i}c(x_1,\dots,x_N)|=|\lambda_i(x_i-B_x)|\leq \lambda_i \diam(\Omega)=\lambda_i D,\quad i=2,\dots,N.
			\end{equation}
			For $(x_2,\dots,x_N), (\tilde{x}_2, \dots,\tilde{x}_N)\in Y^{N-1}$, the triangle inequality gives
			\begin{equation}\label{eq:cost-lipschitz}
				\left|c(x_1, x_2,\dots,x_N)-c(x_1,\tilde{x}_2,\dots,\tilde{x}_N)\right|\leq D \sum_{i=2}^{N}\lambda_i |x_i-\tilde{x}_i|.
			\end{equation}
			Thus,
			\begin{equation}\label{eq:cost-difference-vs-tilde}
				\begin{aligned}
						&\left|\int c(x_1,\dots,x_N)\,\d\hat{\gamma}-\int c(x_1,\dots,x_N)\,\d\tilde\gamma\right|\\
						\overset{\eqref{eq:gluing-marginals}}{=} &\left|\int c(x_1, x_2,\dots,x_N)-c(x_1,\tilde{x}_2,\dots,\tilde{x}_N) \,\d\bar{\gamma}\right|\\
						\overset{\eqref{eq:cost-lipschitz}}{\leq} &D \sum_{i=2}^{N}\lambda_i\int |x_i-\tilde{x}_i|\,\d\bar{\gamma}
						\overset{\eqref{eq:gluing-marginals}}{=} D \delta.
				\end{aligned}
			\end{equation}
			Since $\gamma$ is the optimal transport plan for $(\rho,\mu_2,\dots,\mu_N)$, we have
			\begin{equation}\label{eq:cost-upper-from-gamma}
				\int c(x_1,\dots,x_N)\,\d\gamma\leq \int c(x_1,\dots,x_N)\,\d\hat{\gamma}\leq \int c(x_1,\dots,x_N)\,\d\tilde\gamma+D\delta.
			\end{equation}

			By the same argument with the two target tuples interchanged, there also exists a coupling $\underline{\gamma}\in \Pi(\rho, \mu_2,\dots,\mu_N, \tilde{\mu}_2,\dots,\tilde{\mu}_N)$ with coordinates $(x_1, x_2,\dots, x_N, \tilde x_2,\dots,\tilde x_N)$ such that
			\begin{equation*}
				(x_1, x_2,\dots, x_N)_\#\underline{\gamma}=\gamma,\quad (x_i,\tilde{x}_i)_\#\underline{\gamma}=\pi_i,\quad i=2,\dots,N.
			\end{equation*}
			Denote
			\begin{equation*}
				\check{\gamma}:=(x_1, \tilde x_2,\dots, \tilde x_N)_\#\underline{\gamma}\in \Pi(\rho,\tilde\mu_2, \dots,\tilde\mu_N).
			\end{equation*}
			Similarly, we have
			\begin{equation}\label{eq:cost-upper-from-tilde}
				\int c(x_1,\dots,x_N)\,\d\tilde\gamma\leq \int c(x_1,\dots,x_N)\,\d\check{\gamma}\leq \int c(x_1,\dots,x_N)\,\d\gamma+D\delta.
			\end{equation}

			Combining \eqref{eq:cost-difference-vs-tilde}, \eqref{eq:cost-upper-from-gamma} and \eqref{eq:cost-upper-from-tilde} completes the proof.
		\end{proof}

\subsection{Proof of Theorem \ref{main}}\label{subsec:proof-main}
\begin{proof}[Proof of Theorem \ref{main}]
We first estimate the internal discrepancy along the glued plan. Integrating \eqref{eq:kantorovich-defect-lower} against $\hat\gamma$ gives
\begin{equation}\label{eq:integrated-defect-lower}
 \frac14\int\sum_{i=2}^N\lambda_i^2|V_{i,x}-V_i(x_1)|^2\,\d\hat\gamma
 \leq\int R\,\d\hat\gamma
       +\frac32\int|B_x-B(x_1)|^2\,\d\hat\gamma.
\end{equation}
Since $\hat\gamma$ and $\gamma$ have the same marginals, dual optimality and Lemma~\ref{cost-difference} imply
\begin{equation}\label{eq:defect-cost-difference}
 \int R\,\d\hat\gamma
 =\int c\,\d\hat\gamma-\int c\,\d\gamma
 \leq2D\delta.
\end{equation}
To estimate the remaining barycentric term, write $b:=\|B-\tilde B\|_{L^2(\rho)}$. The triangle inequality and Corollary~\ref{cor:gluing-estimates} then yield
\begin{equation}\label{eq:barycenter-mode-triangle}
 \begin{aligned}
 \left(\int|B_x-B(x_1)|^2\,\d\hat\gamma\right)^{1/2}
 \leq\left(\int|B_x-\tilde B(x_1)|^2\,\d\bar\gamma\right)^{1/2}+b\leq D^{1/2}\delta^{1/2}+b.
 \end{aligned}
\end{equation}
We can now compare the internal components of the two transport maps. Define their weighted discrepancy by
\[
 \mathcal V:=\left(\sum_{i=2}^N\lambda_i^2
                  \|V_i-\tilde V_i\|_{L^2(\rho)}^2\right)^{1/2}.
\]
The triangle inequality in the weighted product $L^2$ space, together with Corollary~\ref{cor:gluing-estimates}, gives
\begin{equation}
 \begin{aligned}
 \mathcal V
 &\leq\left(\int\sum_{i=2}^N\lambda_i^2|V_i(x_1)-V_{i,x}|^2\,\d\hat\gamma\right)^{1/2}
       +D^{1/2}\delta^{1/2}\\
 &\leq\left[8D\delta+6\bigl(D^{1/2}\delta^{1/2}+b\bigr)^2\right]^{1/2}
       +D^{1/2}\delta^{1/2}\\
 &\leq C_1(\delta^{1/2}+b),
 \end{aligned}
\end{equation}
where the second line uses \eqref{eq:integrated-defect-lower}--\eqref{eq:barycenter-mode-triangle}, and $C_1$ depends only on $D$.

Finally, the definitions of the internal components give
\[
 T_i-\tilde T_i=V_i-\tilde V_i+\frac{B-\tilde B}{1-\lambda_1}.
\]
Since $\sum_{i=2}^N\lambda_i=1-\lambda_1$, Cauchy--Schwarz gives
\begin{equation}
 \begin{aligned}
 \sum_{i=2}^N\lambda_i\|T_i-\tilde T_i\|_{L^2(\rho)}
 \leq\sum_{i=2}^N\lambda_i\|V_i-\tilde V_i\|_{L^2(\rho)}+b\leq\sqrt{N-1}\,\mathcal V+b
 \leq C(\delta^{1/2}+b).
 \end{aligned}
\end{equation}
Here $C$ depends only on $N$ and $D$, as required.
\end{proof}

\subsection{Proof of Theorem \ref{multi-marginal}}\label{subsec:proof-multi-marginal}
To control the external mode, we combine Theorem \ref{main} with the following form of M\'erigot's estimate \cite[Theorem 2.2]{merigot:hal-05616391}. For completeness, we derive this form from the convex defect estimate proved below, which also allows us to track its dependence on the source density.
		\begin{proposition}\label{prop:two-marginal-defect}
				Let $S,\, Y\subseteq \R^n$ be compact and let $\Omega=\overline{\operatorname{conv}(S\cup Y)}$. Let $\rho\in\mathcal P_{\mathrm{ac}}(S)\cap L^\infty(S)$. Then there exists a constant $C_2>0$, depending on $\diam(S\cup Y)$, $n$, and $\|\rho\|_{L^\infty}$, such that for any $\mu,\tilde{\mu}\in\mathcal{P}(\Omega)$,
				\begin{equation}
					\|\nabla\phi_\mu-\nabla\phi_{\tilde{\mu}}\|_{L^2(\rho)}^4\leq C_2\int_\Omega (\psi_{\mu}-\psi_{\tilde{\mu}})\,\d (\mu-\tilde{\mu}),
				\end{equation}
			where $\psi_{\mu}$ and $\psi_{\tilde{\mu}}$ are the Kantorovich potentials from $\mu$ to $\rho$ and $\tilde{\mu}$ to $\rho$, respectively, for the quadratic cost; $\phi_{\mu}:=\psi^c_{\mu}$ and $\phi_{\tilde{\mu}}:=\psi^c_{\tilde{\mu}}$ are the corresponding dual Kantorovich potentials. One may take
\[
 C_2=C_n\|\rho\|_{L^\infty}D^{n+2},\qquad D:=\diam(S\cup Y),
\]
with $C_n$ depending only on the dimension; this choice is not asserted to be optimal.
		\end{proposition}
\begin{proof}
Set $Q(x,y)=\frac12|x-y|^2$, and write $T_\mu=\Id-\nabla\phi_\mu$, $T_{\tilde\mu}=\Id-\nabla\phi_{\tilde\mu}$. The two nonnegative dual defects satisfy
\begin{align*}
 &\int\bigl[Q(x,T_{\tilde\mu}(x))-\phi_\mu(x)-\psi_\mu(T_{\tilde\mu}(x))\bigr]\,\d\rho(x)\\
 &\quad+\int\bigl[Q(x,T_\mu(x))-\phi_{\tilde\mu}(x)-\psi_{\tilde\mu}(T_\mu(x))\bigr]\,\d\rho(x)
 =\int(\psi_\mu-\psi_{\tilde\mu})\,\d(\mu-\tilde\mu).
\end{align*}
To estimate the first defect, we apply Lemma~\ref{lem:affine-defect} with $\theta=1$, $v(y)=y$, and $q(y)=\frac12|y|^2-\psi_\mu(y)$ on $\Omega$. The corresponding convex potential is $u(x)=\frac12|x|^2-\phi_\mu(x)$, so the affine defect in that lemma is precisely the first integrand above. In the case $\theta=1$, its proof gives
\[
 \left(\int|v(y)-\nabla u(x)|^2\,\d\Gamma\right)^2
 \leq C_n\|\rho\|_{L^\infty}\diam(S)^{n-1}
        \diam(\Omega)^3\int r\,\d\Gamma.
\]
Both diameters are at most $D$, and the first defect is bounded by the sum of the two nonnegative defects. Applying the preceding estimate to the plan induced by the second transport map and using $T_\mu-T_{\tilde\mu}=\nabla\phi_{\tilde\mu}-\nabla\phi_\mu$ therefore gives the claimed estimate with the stated choice of $C_2$.
\end{proof}

		    \begin{proof}[Proof of Theorem \ref{multi-marginal}]
			By Proposition \ref{prop:barycentermap}, $\mathbb P$ and $\tilde{\mathbb P}$ admit unique Wasserstein barycenters $\mu_{\mathbb P}, \, \mu_{\tilde{\mathbb P}}\in\mathcal{P}(\Omega)$. Moreover, by the barycenter dual optimality condition (cf. \cite{zbMATH05956494}, \cite[Proposition 1.1]{zbMATH07819904}), there exist Kantorovich potentials $\{\psi_{\mu_{{\mathbb P}}}^i\}_{i=1}^N$ such that
			\begin{equation}\label{eq:barycenter-potential-balance}
				\sum_{i=1}^N\lambda_i\psi_{\mu_{{\mathbb P}}}^i\geq0\ \text{on }\Omega,\qquad\sum_{i=1}^N\lambda_i\psi_{\mu_{{\mathbb P}}}^i=0\quad\mu_{\mathbb P}\text{-a.e.},
			\end{equation}
			where $\psi_{\mu_{{\mathbb P}}}^i$ is the Kantorovich potential from $\mu_{\mathbb P}$ to $\mu_i$ for $i=1,\dots,N$
			with $\mu_1=\rho$.
			Similarly, there exist Kantorovich potentials $\{\psi_{\mu_{\tilde{\mathbb P}}}^i\}_{i=1}^N$ satisfying the analogue of \eqref{eq:barycenter-potential-balance}, where $\psi_{\mu_{\tilde{\mathbb P}}}^i$ is the Kantorovich potential from $\mu_{\tilde{\mathbb P}}$ to $\tilde{\mu}_i$ with $\tilde{\mu}_1=\rho$.

			For $\mu_i, \,\tilde{\mu}_i\in\mathcal{P}(\Omega)$, $i=1,\dots,N$, denote
				\begin{equation}\label{2.30}
				D_{{\mu}_i}(\mu_{\tilde{\mathbb P}},\mu_{{\mathbb P}}):=\frac{1}{2}W_2^2(\mu_{\tilde{\mathbb P}},{\mu}_i)-\frac{1}{2}W_2^2(\mu_{{\mathbb P}}, {\mu}_i)-\int_\Omega \psi_{\mu_{{\mathbb P}}}^i\,\d (\mu_{\tilde{\mathbb P}}-\mu_{{\mathbb P}}),
			\end{equation}
			and
			\begin{equation}\label{2.31}
			D_{\tilde{\mu}_i}(\mu_{{\mathbb P}},\mu_{\tilde{\mathbb P}}):=\frac{1}{2}W_2^2(\mu_{{\mathbb P}}, \tilde{\mu}_i)-\frac{1}{2}W_2^2(\mu_{\tilde{\mathbb P}}, \tilde{\mu}_i)-\int_\Omega \psi_{\mu_{\tilde{\mathbb P}}}^i\,\d (\mu_{{\mathbb P}}-\mu_{\tilde{\mathbb P}}).
		\end{equation}

By Kantorovich duality, both defects are nonnegative. Since $\mu_1=\tilde\mu_1=\rho$, the defects corresponding to the source satisfy
\begin{equation}
 \begin{aligned}
 D_\rho(\mu_{\mathbb P},\mu_{\tilde{\mathbb P}})
   +D_\rho(\mu_{\tilde{\mathbb P}},\mu_{\mathbb P})=\int(\psi_{\mu_{\tilde{\mathbb P}}}^1-\psi_{\mu_{\mathbb P}}^1)
                \,\d(\mu_{\tilde{\mathbb P}}-\mu_{\mathbb P})\geq C_2^{-1}\|B-\tilde B\|_{L^2(\rho)}^4.
 \end{aligned}
\end{equation}
Here Proposition~\ref{prop:two-marginal-defect} applies to the source $\rho$, because the quadratic transport maps are
\[
 B=\Id-\nabla\phi_{\mu_{\mathbb P}}^1,\qquad
 \tilde B=\Id-\nabla\phi_{\mu_{\tilde{\mathbb P}}}^1,
 \qquad\phi_{\nu}^1:=(\psi_{\nu}^1)^c.
\]
The balance condition \eqref{eq:barycenter-potential-balance} also gives
			\begin{equation}
				\mathcal{F}_{\mathbb P}(\mu_{\tilde{\mathbb P}})-\mathcal{F}_{\mathbb P}(\mu_{\mathbb P})\geq\sum_{i=1}^{N}\lambda_iD_{{\mu}_i}(\mu_{\tilde{\mathbb P}},\mu_{{\mathbb P}})\geq \lambda_1D_\rho(\mu_{\tilde{\mathbb P}},\mu_{\mathbb P}),
			\end{equation}
			and
			\begin{equation}
			\mathcal{F}_{\tilde{\mathbb P}}(\mu_{\mathbb P})-\mathcal{F}_{\tilde{\mathbb P}}(\mu_{\tilde{\mathbb P}})\geq\sum_{i=1}^{N}\lambda_iD_{\tilde{\mu}_i}(\mu_{{\mathbb P}},\mu_{\tilde{\mathbb P}})\geq \lambda_1 D_\rho(\mu_{\mathbb P}, \mu_{\tilde{\mathbb P}}),
			\end{equation}
			which implies
			\begin{equation}\label{eq:barycenter-energy-lower}
				\begin{aligned}
 \mathcal{F}_{\mathbb P}(\mu_{\tilde{\mathbb P}})-\mathcal{F}_{\mathbb P}(\mu_{\mathbb P})+\mathcal{F}_{\tilde{\mathbb P}}(\mu_{\mathbb P})-\mathcal{F}_{\tilde{\mathbb P}}(\mu_{\tilde{\mathbb P}})
 \geq\lambda_1C_2^{-1}\|B-\tilde B\|_{L^2(\rho)}^4.
 \end{aligned}
			\end{equation}

			On the other hand, for $i=2,\dots,N$, the gluing lemma gives a coupling $\gamma_i\in\Pi(\mu_{\mathbb P}, \mu_i,\tilde{\mu}_i)$ with coordinates $(y, x_i,\tilde{x}_i)$ such that
			\begin{equation*}
				(x_i,\tilde{x}_i)_\#\gamma_i=\pi_i, \quad (y, x_i)_\#\gamma_i=\pi^b_i,
			\end{equation*}
			where $\pi_i\in\Pi(\mu_i,\tilde{\mu}_i)$ is a $W_1$-optimal transport plan between $\mu_i$ and $\tilde{\mu}_i$, and $\pi^b_i\in\Pi(\mu_{\mathbb P}, \mu_i)$ is the $W_2$-optimal transport plan between $\mu_{\mathbb P}$ and $\mu_i$. Denote
			\[
			\tilde{\pi}^b_i:= (y,\tilde{x}_i)_\#\gamma_i\in\Pi(\mu_{\mathbb P},\tilde{\mu}_i).
			\]
			Then we have
			\begin{equation}
				\begin{aligned}
					\frac{1}{2}W_2^2(\mu_{\mathbb P}, \tilde{\mu}_i)-\frac{1}{2}W_2^2(\mu_{\mathbb P}, \mu_i)&\leq \frac{1}{2}\int |y-\tilde{x}_i|^2\,\d\tilde{\pi}^b_i-\frac{1}{2}\int |y-x_i|^2\,\d\pi^b_i\\
					&=\frac{1}{2}\int |y-\tilde{x}_i|^2-|y-x_i|^2\,\d\gamma_i\\
					&\leq D\int |x_i-\tilde{x}_i|\,\d\gamma_i=D W_1(\mu_i,\tilde{\mu}_i).
				\end{aligned}
			\end{equation}
			Hence,
			\begin{equation}\label{eq:energy-perturbation-1}
				\mathcal{F}_{\tilde{\mathbb P}}(\mu_{\mathbb P})-\mathcal{F}_{\mathbb P}(\mu_{\mathbb P})=\sum_{i=2}^{N}\frac{\lambda_i}{2}\left(W_2^2(\mu_{\mathbb P}, \tilde{\mu}_i)-W_2^2(\mu_{\mathbb P}, \mu_i)\right)\leq D\delta.
			\end{equation}
			Similarly, we have
			\begin{equation}\label{eq:energy-perturbation-2}
				\mathcal{F}_{\mathbb P}(\mu_{\tilde{\mathbb P}})-\mathcal{F}_{\tilde{\mathbb P}}(\mu_{\tilde{\mathbb P}})\leq D\delta.
			\end{equation}

			Thus, by \eqref{eq:barycenter-energy-lower}, \eqref{eq:energy-perturbation-1} and \eqref{eq:energy-perturbation-2}, we have
			\begin{equation}
				\|B-\tilde{B}\|_{L^2(\rho)}\leq C_3\lambda_1^{-1/4}\delta^{1/4},
			\end{equation}
			where $C_3$ depends on $\diam(S\cup Y)$, $n$ and $\|\rho\|_{L^\infty}$. Combining this with Theorem \ref{main} gives the estimate.

            \medskip

			  \textbf{Optimality of the exponent:} For $n\geq 2$, let $X,\, Y\subseteq \R^n$ and $\rho$ be the uniform source measure given by \cite[Theorem 2]{letrouit2026unstable} with $p=1$. Denote $S:=\overline{X}$. Then $\rho\in\mathcal P_{\mathrm{ac}}(S)\cap L^\infty(S)$. For any $\alpha>\frac{1}{4}$, there exist $\mu^k, \tilde{\mu}^k\in\mathcal{P}(Y)$ such that, if $T^k$ and $\tilde T^k$ denote the unique optimal transport maps from $\rho$ to $\mu^k$ and from $\rho$ to $\tilde{\mu}^k$ respectively, then
			  \begin{equation}\label{eq:two-marginal-instability}
			  	\frac{\|T^k-\tilde T^k\|_{L^2(\rho)}}{W_1^\alpha(\mu^k,\tilde{\mu}^k)}\rightarrow+\infty.
			  \end{equation}

			   For each $k\in\N$, define $\mu_2^k=\dots=\mu_N^k=\mu^k$, $\tilde{\mu}_2^k=\dots=\tilde{\mu}_N^k=\tilde{\mu}^k$. Then for any $\pi\in\Pi(\rho,\mu_2^k,\dots,\mu_N^k)$, by \eqref{equal}, we have
			  \begin{equation}
			  	\int c(x_1,\dots,x_N)\,\d\pi\geq \frac{1}{2}\sum_{i=2}^{N}\lambda_1\lambda_i\int |x_1-x_i|^2\,\d\pi\geq \frac{1}{2}\lambda_1(1-\lambda_1)W_2^2(\rho,\mu^k).
			  \end{equation}

			  By \cite{zbMATH01060715} (see also \cite{zbMATH06484002}), equality holds if and only if $\pi=(\Id, T^k,\dots, T^k)_\#\rho$, which implies that the multi-marginal optimal transport maps associated with $(\rho,\mu_2^k,\dots, \mu_N^k)$ and $(\rho,\tilde{\mu}_2^k,\dots, \tilde{\mu}_N^k)$ are $(T^k,\dots,T^k)$ and $(\tilde{T}^k,\dots,\tilde{T}^k)$. Therefore,
			  \begin{equation}
			  	\sum_{i=2}^N\lambda_i\|T_i^k-\tilde T_i^k\|_{L^2(\rho)}=(1-\lambda_1)\|T^k-\tilde{T}^k\|_{L^2(\rho)}, \quad \sum_{i=2}^N \lambda_iW_1(\mu_i^k,\tilde\mu_i^k)=(1-\lambda_1)W_1(\mu^k,\tilde{\mu}^k).
			  \end{equation}
		      Combining this with \eqref{eq:two-marginal-instability} proves the optimality of the exponent.
		\end{proof}

\subsection{Proof of Theorem \ref{self-comparison}}\label{subsec:proof-self-comparison}
The optimality of the exponents is proved in \S\ref{appendix:optimality}.
\begin{proof}[Proof of Theorem \ref{self-comparison}]
Denote
\begin{equation*}
    \tilde{\mu}_i:=\mu_{\sigma(i)},\quad \tilde T_i:=T_{\sigma(i)},\quad i=2,\dots,N.
\end{equation*}
Since $\lambda_{\sigma(i)}=\lambda_i$, the barycentric quadratic cost $c(x_1,\dots,x_N)$ is invariant under the permutation $\sigma$. Hence $(\tilde T_2,\dots,\tilde T_N)$ are the multi-marginal optimal transport maps associated with $(\rho,\tilde\mu_2,\dots,\tilde\mu_N)$. Moreover,
\begin{equation}
    \tilde B=\lambda_1\Id+\sum_{i=2}^N\lambda_i\tilde T_i
    =\lambda_1\Id+\sum_{i=2}^N\lambda_iT_{\sigma(i)}
    =\lambda_1\Id+\sum_{i=2}^N\lambda_{\sigma(i)}T_{\sigma(i)}=B.
\end{equation}
Corollary \ref{cor:barycenter-preserving-stability} therefore gives
\begin{equation}
    \sum_{i=2}^N\lambda_i\|T_i-T_{\sigma(i)}\|_{L^2(\rho)}
    =\sum_{i=2}^N\lambda_i\|T_i-\tilde T_i\|_{L^2(\rho)}
    \le C\left(\sum_{i=2}^N\lambda_i W_1(\mu_i,\mu_{\sigma(i)})\right)^{1/2},
\end{equation}
where $C$ depends only on $N$ and $\diam(S\cup Y)$.

If $\sigma$ is the transposition exchanging $i$ and $j$, then
\begin{equation}
    \sum_{k=2}^N\lambda_k\|T_k-T_{\sigma(k)}\|_{L^2(\rho)}=(\lambda_i+\lambda_j)\|T_i-T_j\|_{L^2(\rho)},
\end{equation}
and
\begin{equation}
    \sum_{k=2}^N\lambda_k W_1(\mu_k,\mu_{\sigma(k)})=(\lambda_i+\lambda_j)W_1(\mu_i,\mu_j).
\end{equation}
This proves the pairwise estimate.
\end{proof}

\begin{proof}[Proof of Corollary \ref{cor:pairwise-two-scale}]
	 By Proposition \ref{prop:barycentermap} and \cite[Theorem 5.1]{zbMATH05956494} (see also \cite[Theorem 5.1]{zbMATH06670887}), the unique Wasserstein barycenter $\mu_{{\mathbb P}}\in\mathcal{P}(\Omega)$ satisfies the $L^\infty$ estimate
	 \begin{equation}
	 	\|\mu_{\mathbb P}\|_{L^\infty}\le \lambda_1^{-n}\|\rho\|_{L^\infty}.
	 \end{equation}
 Moreover, by \cite[Proposition 5.1]{zbMATH06484002}, we have
\begin{equation*}
    T_i=T_{\mu_{\mathbb P}}^i\circ B,\quad T_j=T_{\mu_{\mathbb P}}^j\circ B,
\end{equation*}
where $T_{\mu_{\mathbb P}}^i$ and $T_{\mu_{\mathbb P}}^j$ are the unique optimal transport maps from $\mu_{\mathbb P}$ to $\mu_i$ and from $\mu_{\mathbb P}$ to $\mu_j$, respectively. Then M\'erigot's two-marginal stability estimate yields
\begin{equation}
    \|T_i-T_j\|_{L^2(\rho)}
    =\|T_{\mu_{\mathbb P}}^i-T_{\mu_{\mathbb P}}^j\|_{L^2(\mu_{\mathbb P})}
    \le C W_1^{1/4}(\mu_i,\mu_j).
\end{equation}
To make the dependence of the constant explicit, we apply Proposition~\ref{prop:two-marginal-defect} with source $\mu_{\mathbb P}$. The difference of the two target-side quadratic potentials is $2D$-Lipschitz on $Y$, and Kantorovich--Rubinstein duality therefore gives
\[
 \begin{aligned}
 \|T_i-T_j\|_{L^2(\rho)}^4
 \leq C_nD^{n+3}\|\mu_{\mathbb P}\|_{L^\infty}W_1(\mu_i,\mu_j)\leq C_nD^{n+3}\lambda_1^{-n}\|\rho\|_{L^\infty}W_1(\mu_i,\mu_j).
 \end{aligned}
\]
The other branch follows from the proof of Theorem~\ref{self-comparison} with $B=\tilde B$ and has constant $C_ND^{1/2}$. Taking the larger of these two constants gives the displayed choice of $C$ in the corollary. To obtain the interpolation estimate, set $w=\lambda_i+\lambda_j$ and use
\[
 \min\{X,Y\}\leq X^{2-4\theta}Y^{4\theta-1},\qquad
 X=W_1^{1/4},\quad Y=w^{-1/2}W_1^{1/2},
\]
when the distance is positive; the zero-distance case follows directly from either bound.
\end{proof}

\subsection{Sharpness and weight dependence}\label{appendix:optimality}
To prove the optimality statement in Proposition~\ref{exponent}, we use a fixed-source construction inspired by \cite[Theorem~2]{letrouit2026unstable}. The construction shows how uniform stability fails as target weights tend to zero. We explain the resulting scale estimates here and give the geometric construction and its verification in Appendix~\ref{app:fixed-source-construction}.

\begin{proof}[Proof of Proposition~\ref{exponent}]
For \textup{(a)}, take the uniform source on $S=[0,1]^n$, weights $\lambda_i=1/N$, and targets
\[
 \mu_2^\epsilon=(1-\epsilon)\delta_0+\epsilon\delta_{e_1},
 \qquad \mu_3=\cdots=\mu_N=\delta_0,\qquad 0<\epsilon<1.
\]
The pushforward constraints alone give $T_3=0$ and
\[
 \|T_2^\epsilon-T_3\|_{L^2(\rho)}=\epsilon^{1/2},
 \qquad W_1(\mu_2^\epsilon,\mu_3)=\epsilon.
\]
Swapping the second and third marginals preserves the barycenter, so the same example also rules out improving the square-root exponent for barycenter-preserving perturbations. In particular, taking $n=1$ rules out a larger exponent for general perturbations on the line, as asserted in Remark~\ref{rem:one-dim}.

For \textup{(b)}, Lemma~\ref{lem:weight-construction} gives fixed $S,Y,\rho$ and target tuples with
\[
 \lambda_2^k=\cdots=\lambda_N^k=\lambda_k:=2^{-(k+k_0)^2},
 \qquad \lambda_1^k=1-(N-1)\lambda_k\geq\tfrac12,
 \qquad s_k=\eta2^{-k},
\]
such that
\begin{equation}\label{eq:weight-scales}
 \|T_2^k-T_3^k\|_{L^2(\rho)}\asymp\lambda_k^{1/2}s_k^{(n+2)/2},
 \qquad W_1(\mu_2^k,\mu_3^k)\asymp\lambda_k^2s_k^{n+1}.
\end{equation}
The source is supported on separated cells accumulating at the origin. Within the $k$th cell, the two target pairs have a small horizontal separation, but their optimal images of the same source point are vertically separated by a distance of order $s_k$. The cells are sufficiently far apart to prevent transport between them, while the choices of $\lambda_k$ and $s_k$ allow all comparisons to be made with the same source measure.

Consequently, for $\theta\geq\frac14$ and $\gamma>\frac12-2\theta$,
\begin{equation}\label{eq:weight-tradeoff-obstruction}
 \frac{\|T_2^k-T_3^k\|_{L^2(\rho)}}{(2\lambda_k)^\gamma W_1^\theta(\mu_2^k,\mu_3^k)}
 \asymp\lambda_k^{1/2-2\theta-\gamma}
 s_k^{(n+2)/2-(n+1)\theta}\longrightarrow+\infty.
\end{equation}
Indeed, for every $a>0$ and $b\in\R$,
\[
 \lambda_k^{-a}s_k^b=\eta^b2^{a(k+k_0)^2-bk}\longrightarrow+\infty.
\]
Taking $a=2\theta+\gamma-\frac12>0$ proves \textup{(b)}, since all comparison constants are independent of $k$ and $S,Y,\rho$ remain fixed.
\end{proof} 

\section{Extensions of the Two-Mode Estimates}\label{sec:beyond}
The preceding arguments rely on two distinct estimates: the Kantorovich defect controls relative displacements, and a separate stability estimate controls the remaining collective coordinate. This section examines how these two parts of the argument apply beyond the basic barycentric model. We begin in Section~\ref{sec:barycenter-preserving} with two non-permutation examples within that model. In Section~\ref{subsec:collective-coordinate}, we show that perturbing the cost through the collective coordinate preserves the relative estimate and leads to a penalized barycenter problem for the external mode. Section~\ref{subsec:pairwise-quadratic} then treats pairwise quadratic interactions under an inertia condition, including singular sources and approximately optimal transport plans.

\subsection{Barycenter-preserving perturbations}\label{sec:barycenter-preserving}
			We present two non-permutation examples that preserve the barycenter map. The construction relies
			on the following standard dual characterization of Wasserstein barycenters (cf. \cite{zbMATH05956494}); see also \cite{zbMATH07007327} for a related averaging-map characterization.
			\begin{proposition}\label{prop:barycenter-dual-characterization}
				Let $\mathbb P=\lambda_1\delta_\rho+\sum_{i=2}^N\lambda_i\delta_{\mu_i}\in\mathcal{P}(\mathcal{P}(\Omega))$. Then the following statements are equivalent:
				\begin{itemize}
					\item[\textup{\textbf{(a)}}] $\nu\in\mathcal{P}(\Omega)$ is a Wasserstein barycenter of $\mathbb P$;
					\item[\textup{\textbf{(b)}}] There exist Kantorovich potentials $\Psi_i$ from $\nu$ to $\mu_i$, for $i=1,\dots,N$ with $\mu_1=\rho$, after choosing suitable additive constants, such that
					\begin{equation*}
						\sum_{i=1}^{N}\lambda_i \Psi_i(x)\geq 0, \quad \forall \, x\in\Omega,\quad\sum_{i=1}^{N}\lambda_i \Psi_i(x)=0,\quad \nu \text{-a.e.} \, x\in\Omega.
					\end{equation*}
				\end{itemize}
			\end{proposition}

			In the following two examples, fix $\mathbb P=\lambda_1\delta_\rho+\sum_{i=2}^N\lambda_i\delta_{\mu_i}\in\mathcal{P}(\mathcal{P}(\Omega))$ and assume $\rho\in\mathcal P_{\mathrm{ac}}(\Omega)$. By Proposition \ref{prop:barycentermap}, there exists a unique Wasserstein barycenter $\mu_{\mathbb P}$ of $\mathbb P$ with $\mu_{\mathbb P}\ll \mathcal{L}^n$. For $i=1,\dots,N$, let $T_{\mu_{\mathbb P}}^i=\Id-\nabla \Psi_{\mu_{\mathbb P}}^i$ be the unique optimal transport map from $\mu_{\mathbb P}$ to $\mu_i$ with $\mu_1=\rho$, where $\Psi_{\mu_{\mathbb P}}^i$ is a Kantorovich potential from $\mu_{\mathbb P}$ to $\mu_i$. We choose the additive constants of $\Psi_{\mu_{\mathbb P}}^i$ such that Proposition \ref{prop:barycenter-dual-characterization} \textup{\textbf{(b)}} holds.

In both examples, the perturbed potentials satisfy
\[
 \tilde\Psi_{\mu_{\mathbb P}}^1=\Psi_{\mu_{\mathbb P}}^1,
 \qquad \sum_i\lambda_i\tilde\Psi_{\mu_{\mathbb P}}^i
 =\sum_i\lambda_i\Psi_{\mu_{\mathbb P}}^i.
\]
Provided that the functions $\frac12|x|^2-\tilde\Psi_{\mu_{\mathbb P}}^i$ are convex and the induced target measures remain supported in $\Omega$, the maps $\Id-\nabla\tilde\Psi_{\mu_{\mathbb P}}^i$ are optimal from $\mu_{\mathbb P}$ to those measures. The balance identity then allows us to apply Proposition~\ref{prop:barycenter-dual-characterization}, showing that the barycenter is unchanged. Since the first potential is also unchanged, the source measure is preserved. It therefore remains to verify convexity and the balance identity in each example.

		   \begin{example}[Balanced translations]
		   	Let $a_1=0\in\R^n$, and let $a_2,\dots,a_N\in\R^n$ be such that $\sum_{i=2}^N\lambda_i a_i=0$.
		   	Define
		   	\begin{equation*}
		   		\tilde{\Psi}_{\mu_{\mathbb P}}^i(x):=\Psi_{\mu_{\mathbb P}}^i(x)-\langle a_i, x\rangle, \quad i=1,\dots,N,
		   	\end{equation*}
		   	and
		   	\begin{equation*}
		   		\tilde{T}_{\mu_{\mathbb P}}^i:=\Id-\nabla \tilde{\Psi}_{\mu_{\mathbb P}}^i=T_{\mu_{\mathbb P}}^i+a_i, \quad \tilde{\mu}_i:=(\tilde{T}_{\mu_{\mathbb P}}^i)_{\#}\mu_{\mathbb P}.
		   	\end{equation*}

		   Assume the translations are chosen such that $\spt \tilde{\mu}_i\subseteq \Omega$ for $i=1,\dots,N$. The identity
		   	\begin{equation}
		   			\frac{1}{2}|x|^2-\tilde{\Psi}_{\mu_{\mathbb P}}^i(x)=	\left(\frac{1}{2}|x|^2-\Psi_{\mu_{\mathbb P}}^i(x)\right)+\langle a_i,x\rangle, \quad i=1,\dots, N
		   	\end{equation}
		   	shows that the perturbed quadratic potentials remain convex. Moreover,
		   	\begin{equation}
		   		\sum_{i=1}^{N}\lambda_i \tilde{\Psi}_{\mu_{\mathbb P}}^i(x)=\sum_{i=1}^{N}\lambda_i \Psi_{\mu_{\mathbb P}}^i(x)-	\left\langle \sum_{i=1}^N\lambda_i a_i,x\right\rangle= \sum_{i=1}^{N}\lambda_i \Psi_{\mu_{\mathbb P}}^i(x),
		   	\end{equation}
		   	so the preceding criterion shows that the perturbed tuple preserves the barycenter map.
		   \end{example}

		   \begin{example}[Balanced nonlinear perturbations]
		   	Let $k,l\in \{2,\dots,N\}$ with $k\neq l$. Let $h\in C^2_c(\mathbb R^n)$, and let $\epsilon>0$ be such that
		   	\begin{equation}\label{eq:nonlinear-admissibility}
		   			\frac12|x|^2-\Psi_{\mu_{\mathbb P}}^k(x)+\epsilon h(x), \quad
		   		\frac12|x|^2-\Psi_{\mu_{\mathbb P}}^l(x)-\epsilon\frac{\lambda_k}{\lambda_l}h(x)
		   	\end{equation}
		   	remain convex and the perturbed measures $\tilde{\mu}_i$ defined below satisfy $\spt \tilde{\mu}_i\subseteq \Omega$ for $i=1,\dots,N$. This is an admissibility assumption on the pair $(h,\epsilon)$. Define
		   	\begin{align*}
\tilde \Psi_{\mu_{\mathbb P}}^k(x)
    :=\Psi_{\mu_{\mathbb P}}^k(x)-\epsilon h(x), \quad
\tilde \Psi_{\mu_{\mathbb P}}^l(x)
    :=\Psi_{\mu_{\mathbb P}}^l(x)+\epsilon\frac{\lambda_k}{\lambda_l}h(x),
\end{align*}
\[\tilde \Psi_{\mu_{\mathbb P}}^i(x)
:=\Psi_{\mu_{\mathbb P}}^i(x), \quad i\neq k,l,\]
		   	and
		   	\begin{align*}
\tilde{T}_{\mu_{\mathbb P}}^k
    :=\Id-\nabla \tilde{\Psi}_{\mu_{\mathbb P}}^k
      =T_{\mu_{\mathbb P}}^k+\epsilon\nabla h, \quad
\tilde{T}_{\mu_{\mathbb P}}^l
    :=\Id-\nabla \tilde{\Psi}_{\mu_{\mathbb P}}^l
      =T_{\mu_{\mathbb P}}^l-\epsilon\frac{\lambda_k}{\lambda_l}\nabla h,
\end{align*}
\[\tilde{T}_{\mu_{\mathbb P}}^i
:={T}_{\mu_{\mathbb P}}^i, \quad i\neq k,l.\]
		   	Define
		   	\begin{equation*}
		   		\tilde{\mu}_i:= (\tilde{T}_{\mu_{\mathbb P}}^i)_{\#}\mu_{\mathbb P}, \quad i=1,\dots,N.
		   	\end{equation*}

           Convexity holds by the admissibility assumption, and the balance identity is
		   	\begin{equation}
		   			\sum_{i=1}^{N}\lambda_i \tilde{\Psi}_{\mu_{\mathbb P}}^i(x)=\sum_{i=1}^{N}\lambda_i \Psi_{\mu_{\mathbb P}}^i(x)-\lambda_k\epsilon h(x)+\lambda_l \epsilon\frac{\lambda_k}{\lambda_l}h(x)= \sum_{i=1}^{N}\lambda_i \Psi_{\mu_{\mathbb P}}^i(x).
		   	\end{equation}
		   	The criterion above again shows that the perturbed tuple preserves the barycenter map.
		   \end{example}

		   \begin{remark}
		   	The admissibility assumption in the nonlinear example ensures that the functions in \eqref{eq:nonlinear-admissibility} remain convex. For the diagonal tuple $\mu_2=\dots=\mu_N=\rho$, convexity holds for sufficiently small $\epsilon$; the support condition can also be met by taking $h$ supported in the interior of $\Omega$. In both examples, the linear balance condition on the Kantorovich potentials preserves the barycenter, while convexity determines the admissible perturbations.
		   \end{remark}

\subsection{Collective-coordinate costs}\label{subsec:collective-coordinate}
The two-mode argument continues to apply when the cost is modified through its collective coordinate, but identifying the external transport map requires an additional step. Proposition~\ref{prop:collective-coercivity} first bounds the internal displacements in terms of the collective map $Z^F$. Proposition~\ref{prop:collective-full-stability} then identifies $H^F=Z^F-\nabla F(Z^F)$ as an optimal transport map to a penalized Wasserstein barycenter. We can thus apply the two-marginal defect estimate to $H^F$ and use strong monotonicity to transfer the resulting bound to $Z^F$.

Let \(\Omega=\overline{\operatorname{conv}(S\cup Y)}\). For
\((x_1,\dots,x_N)\in S\times Y^{N-1}\), define the
linear collective coordinate
\[
    Z_x:=\sum_{i=1}^N\lambda_i x_i.
\]
For the quadratic cost \eqref{cost}, this coordinate is exactly the pointwise
weighted barycenter. We then define
\[
    M_x:=\frac{Z_x-\lambda_1x_1}{1-\lambda_1},
    \quad
    V_{i,x}:=x_i-M_x,
    \quad i=2,\dots,N.
\]

Let \(c_0\) denote the weighted quadratic cost in \eqref{cost}. Consider the cost
\begin{equation}\label{eq:collective-cost}
    c_F(x_1,\dots,x_N):=c_0(x_1,\dots,x_N)+F(Z_x),
\end{equation}
where \(F\in C^2(U)\) for some open convex neighborhood \(U\supseteq\Omega\). Then
\begin{equation}\label{eq:collective-splitting}
    c_F(x_1,\dots,x_N)
    =\frac{\lambda_1}{2(1-\lambda_1)}|Z_x-x_1|^2
      +\frac12\sum_{i=2}^N\lambda_i|V_{i,x}|^2
      +F(Z_x).
\end{equation}

Assume that there exist constants
\(\kappa_F\in(0,1]\) and \(\tau_F\ge0\) such that, as quadratic forms on $U$,
\begin{equation}\label{eq:F-curvature}
    -\tau_F I \leq D^2F\leq (1-\kappa_F)I.
\end{equation}
The Kim--Pass twist condition follows from the monotonicity of the transformed collective coordinate. Indeed, the map
\[
    H(z):=z-\nabla F(z)
\]
satisfies
\[
    \langle H(z)-H(z'),z-z'\rangle
    \ge \kappa_F |z-z'|^2,
    \quad z,z'\in U.
\]
Thus \(H\) is injective. Moreover,
\[
    D_{x_1}c_F(x_1,\dots,x_N)
    =
    \lambda_1(x_1-H(Z_x)).
\]
Hence, for fixed $x_1$, equality of the source derivatives at two points of a splitting set forces their collective coordinates to agree. To recover the individual coordinates, put $G(z)=\frac12|z|^2-F(z)$, so $D^2G\geq\kappa_F I$. If two contact tuples have the same collective coordinate and differ by $\Delta x_i$ in their target coordinates, the one-coordinate dual inequalities give
\[
 0\geq\sum_{i=2}^N\bigl[G(Z_x+\lambda_i\Delta x_i)-G(Z_x)
             -\lambda_i\langle\nabla G(Z_x),\Delta x_i\rangle\bigr]
 \geq\frac{\kappa_F}{2}\sum_{i=2}^N\lambda_i^2|\Delta x_i|^2.
\]
The lower bound forces the two tuples to coincide, which verifies twist on splitting sets for $c_F$ (cf.\ \cite{kim2014general}). By Kim and Pass \cite[Theorem 3.1]{kim2014general}, for $\rho\in\mathcal P_{\mathrm{ac}}(S)$, there exists a unique tuple of multi-marginal optimal transport maps \((T^F_2,\dots,T^F_N)\) associated with \((\rho,\mu_2,\dots,\mu_N)\).

Define the collective map, the associated target mean and the internal directions by
\[Z^F(x_1):=\lambda_1x_1+\sum_{i=2}^N\lambda_iT_i^F(x_1),\]
and
\begin{equation*}
	M^F(x_1):=\frac{Z^F(x_1)-\lambda_1x_1}{1-\lambda_1},\quad
	V_i^F(x_1):=T_i^F(x_1)-M^F(x_1), \quad i=2,\dots,N.
\end{equation*}
Let \((\varphi^F_1,\dots,\varphi^F_N)\) be a \(c_F\)-conjugate tuple of Kantorovich potentials for the multi-marginal optimal transport problem \((\rho,\mu_2,\dots,\mu_N)\). For $x_1\in S$, $x_2,\dots, x_N\in Y$, define
\[
R_F(x_1,\dots,x_N)
:=
c_F(x_1,\dots,x_N)-\sum_{i=1}^N\varphi^F_i(x_i)\geq 0,
\]
and
\[
    R_F(x_1,T_2^F(x_1), \dots, T_N^F(x_1))=0, \quad \rho\text{-a.e. }x_1\in S.
\]
\begin{proposition}\label{prop:collective-coercivity} Let $\rho\in\mathcal P_{\mathrm{ac}}(S)$ and let \(c_F\) be
	defined by \eqref{eq:collective-cost}.
For \(\rho\)-a.e. \(x_1\in S\) and every \(x_2,\dots,x_N\in Y\), we have
\begin{equation}\label{eq:collective-residual-bound}
    R_F(x_1,\dots,x_N)
    \ge
    \frac{\kappa_F}{4}\sum_{i=2}^N \lambda_i^2
    |V_{i,x}-V^F_i(x_1)|^2
    -
    \frac{1+\tau_F+2\kappa_F}{2}|Z_x-Z^F(x_1)|^2.
\end{equation}

Moreover, let \((\tilde\mu_2,\ldots,\tilde\mu_N)\subseteq\mathcal P(Y)\) be another tuple of target marginals, and let
\((\tilde T_2^F,\ldots,\tilde T_N^F)\) be the corresponding multi-marginal optimal transport maps. Define
\[
    \tilde Z^F(x_1):=\lambda_1x_1+\sum_{i=2}^N\lambda_i\tilde T_i^F(x_1).
\]
Then
\begin{equation}\label{eq:collective-stability}
    \sum_{i=2}^N \lambda_i
    \|T^F_i-\tilde T^F_i\|_{L^2(\rho)}
    \le
    C
    \left[
    \left(
    \sum_{i=2}^N\lambda_iW_1(\mu_i,\tilde\mu_i)
    \right)^{1/2}
    +
    \|Z^F-\tilde Z^F\|_{L^2(\rho)}
    \right],
\end{equation}
where \(C\) depends on \(N\), \(\diam(S\cup Y)\),
\(\kappa_F\) and \(\tau_F\).
\end{proposition}

\begin{proof}
	Fix $x_1\in S$ such that
	\begin{equation}\label{eq:collective-defect-zero}
		R_F(x_1,T_2^F(x_1),\dots,T_N^F(x_1))=0.
	\end{equation}
For $i=2,\dots,N$, denote
\[
    \Delta x_i:=x_i-T_i^F(x_1),
    \quad
    \Delta Z:=Z_x-Z^F(x_1)=\sum_{i=2}^N\lambda_i\Delta x_i, \quad \Delta V_i:=V_{i,x}-V^F_i(x_1).
\]
Then
\begin{equation}\label{eq:collective-delta-v}
	  \Delta V_i
	=
	\Delta x_i-\frac{\Delta Z}{1-\lambda_1},
	\quad i=2,\dots,N,\quad \text{and}\quad
	\sum_{i=2}^N\lambda_i\Delta V_i=0.
\end{equation}

For $h\in\mathbb R^n$ with $Z^F(x_1)+h\in U$, define the Taylor remainder
\[
\mathcal R_F(h)
:=
F(Z^F(x_1)+h)-F(Z^F(x_1))-\langle \nabla F(Z^F(x_1)),h\rangle.
\]
By Taylor's formula and
\eqref{eq:F-curvature}, we have
\begin{equation}\label{eq:taylor-remainder-bounds}
	\mathcal R_F(\Delta Z)
	\ge
	-\frac{\tau_F}{2}|\Delta Z|^2,\quad \mathcal R_F(\lambda_i\Delta x_i)\le
	\frac{1-\kappa_F}{2}\lambda_i^2|\Delta x_i|^2, \quad  i=2,\dots,N.
\end{equation}

Let $r_i^F$ be the nonnegative defect obtained from the contact tuple by changing only coordinate $i$. The same cancellation of marginal potentials as in Proposition~\ref{prop:kantorovich-defect-lower} gives the exact identity
\begin{equation}\label{eq:collective-defect-expansion}
 \begin{aligned}
 R_F=\sum_{i=2}^Nr_i^F+\frac12\sum_{i=2}^N\lambda_i^2|\Delta x_i|^2
       -\frac12|\Delta Z|^2+\mathcal R_F(\Delta Z)-\sum_{i=2}^N\mathcal R_F(\lambda_i\Delta x_i).
 \end{aligned}
\end{equation}
Since all Taylor remainders are evaluated along segments contained in $\Omega\subset U$, the curvature bounds in \eqref{eq:taylor-remainder-bounds} apply to every term. Together with the nonnegativity of the single-coordinate defects, they imply
\begin{equation}\label{eq:collective-coordinate-bound}
 R_F\geq\frac{\kappa_F}{2}\sum_{i=2}^N\lambda_i^2|\Delta x_i|^2
            -\frac{1+\tau_F}{2}|\Delta Z|^2.
\end{equation}
To express this estimate in terms of internal displacements, introduce
\[
 w:=1-\lambda_1,\qquad q_\lambda:=\sum_{i=2}^N\lambda_i^2\leq w^2,
 \qquad A_V:=\sum_{i=2}^N\lambda_i^2|\Delta V_i|^2.
\]
The relation \eqref{eq:collective-delta-v} separates each displacement into its internal and collective parts, giving the expansion
\[
 \sum_{i=2}^N\lambda_i^2|\Delta x_i|^2
 =A_V+\frac{2}{w}\left\langle\sum_{i=2}^N\lambda_i^2\Delta V_i,\Delta Z\right\rangle
      +\frac{q_\lambda}{w^2}|\Delta Z|^2.
\]
Cauchy--Schwarz gives $|\sum_i\lambda_i^2\Delta V_i|\leq q_\lambda^{1/2}A_V^{1/2}\leq wA_V^{1/2}$. Discarding the last, nonnegative term in the expansion and applying Young's inequality in the form
\[
 \kappa_FA_V^{1/2}|\Delta Z|
 \leq\frac{\kappa_F}{4}A_V+\kappa_F|\Delta Z|^2
\]
therefore yields
\[
 \frac{\kappa_F}{2}\sum_{i=2}^N\lambda_i^2|\Delta x_i|^2
 \geq\frac{\kappa_F}{4}A_V-\kappa_F|\Delta Z|^2.
\]
Substituting this bound into \eqref{eq:collective-coordinate-bound} proves \eqref{eq:collective-residual-bound}. This application of Young's inequality retains half of the internal coefficient and contributes the additional term that gives the collective coefficient $(1+\tau_F)/2+\kappa_F$.

We now derive the stability estimate. By Lemma \ref{lem:gluing-coupling}, there exists a coupling $\bar\gamma
\in
\Pi(\rho,\mu_2,\dots,\mu_N,\tilde\mu_2,\dots,\tilde\mu_N)$ such that
\[
(x_1,\tilde x_2,\ldots,\tilde x_N)_\#\bar\gamma
=
(\Id,\tilde T_2^F,\ldots,\tilde T_N^F)_\#\rho,\quad
(x_i,\tilde x_i)_\#\bar\gamma=\pi_i,\quad i=2,\ldots,N,
\]
where $\pi_i\in\Pi(\mu_i,\tilde{\mu}_i)$ is a $W_1$-optimal transport plan between $\mu_i$ and $\tilde{\mu}_i$.
Recall that $\delta=\sum_{i=2}^N\lambda_iW_1(\mu_i,\tilde\mu_i)$ and $D=\diam(S\cup Y)$. By Corollary \ref{cor:gluing-estimates}, with \(B\) replaced by \(Z\), we have
\begin{equation}\label{eq:Z-gluing}
 \begin{aligned}
 \left(\int|Z_x-\tilde Z^F(x_1)|^2\,\d\bar\gamma\right)^{1/2}
 &\leq D^{1/2}\delta^{1/2},\\
 \left(\int\sum_{i=2}^N\lambda_i^2|V_{i,x}-\tilde V_i^F(x_1)|^2\,\d\bar\gamma\right)^{1/2}
 &\leq D^{1/2}\delta^{1/2}.
 \end{aligned}
\end{equation}

Moreover, fix $z_0\in\Omega$ and define
\begin{equation*}
	\bar F(z)
	:=
	F(z)-F(z_0)-\langle\nabla F(z_0),z-z_0\rangle.
\end{equation*}
Since \(\Omega\) is convex, $D^2\bar{F}=D^2F$ and $\nabla \bar F(z_0)=0$. By \eqref{eq:F-curvature}, we have
\begin{equation}\label{eq:barF-lipschitz}
	\Lip(\bar F\restr\Omega)
	\leq
	D\max\{\tau_F,1-\kappa_F\}.
\end{equation}

Also,
\begin{equation}
	c_0(x_1,\dots,x_N)+\bar F(Z_x)
	=
	c_F(x_1,\dots,x_N)
	-
	\sum_{i=1}^N\lambda_i
	\langle \nabla F(z_0),x_i\rangle
	+
	\langle \nabla F(z_0),z_0\rangle
	-
	F(z_0),
\end{equation}
so the multi-marginal optimal transport maps are unchanged. Moreover, if the
Kantorovich potentials are shifted by the corresponding marginal affine
functions, then the Kantorovich defect
$c_F-\sum_{i=1}^N\varphi_i^F$
is unchanged as well. Hence the stability argument may be carried out with
\(\bar F\) in place of \(F\). With $c_F$ denoting the cost after this affine normalization in the next display, \eqref{eq:barF-lipschitz} gives
\begin{equation}
	\label{eq:cF-Lipschitz}
	|c_F(x_1,x_2,\dots,x_N)-c_F(x_1,\tilde{x}_2,\dots,\tilde{x}_N)|
	\leq C_5\sum_{i=2}^N\lambda_i|x_i-\tilde{x}_i|,
\end{equation}
where \(C_5\) depends on \(D,\tau_F,\kappa_F\).

Using the analogue of Lemma \ref{cost-difference}, with \eqref{eq:cF-Lipschitz} in place of \eqref{eq:cost-lipschitz}, and repeating the proof of
Theorem \ref{main}, we obtain the stability estimate \eqref{eq:collective-stability}.
\end{proof}

The estimate in \eqref{eq:collective-stability} leaves the collective term
$\|Z^F-\tilde Z^F\|_{L^2(\rho)}$ to be controlled. For \(F=0\), the map \(Z^F\) is the Wasserstein barycenter map used in
Theorem \ref{multi-marginal}. For a general \(F\), \(Z^F\) need not be an optimal transport map. The external mode is described by the transformed collective map $H^F:=Z^F-\nabla F(Z^F)$, which can be identified with an optimal transport map associated with a penalized Wasserstein barycenter problem. The following proposition makes this reduction precise.
\begin{proposition}\label{prop:collective-full-stability}
	Let $\rho\in\mathcal P_{\mathrm{ac}}(S)\cap L^\infty(S)$ and let \(c_F\) be defined by \eqref{eq:collective-cost}. With the notation of Proposition \ref{prop:collective-coercivity}, there exists a constant $C>0$, depending on $\lambda_1, N, \diam(S\cup Y),\kappa_F, \tau_F, n$ and $\|\rho\|_{L^\infty}$, such that
	\begin{equation*}
		\sum_{i=2}^N\lambda_i\|T_i^F-\tilde T_i^F\|_{L^2(\rho)}
		\leq C\left(\sum_{i=2}^N \lambda_iW_1(\mu_i,\tilde\mu_i)\right)^{1/4}.
	\end{equation*}
\end{proposition}
\begin{proof}
	As in the proof of Proposition \ref{prop:collective-coercivity}, we fix $z_0\in\Omega$ and replace \(F\) by the normalized function
	\begin{equation*}
		\bar F(z)
		:=
		F(z)-F(z_0)-\langle\nabla F(z_0),z-z_0\rangle.
	\end{equation*}
	This does not change the multi-marginal optimal transport maps. After relabeling
	\(\bar F\) as \(F\), we may assume that \(\nabla F(z_0)=0\).

	Let \(K\) be a compact convex set such that
\(\Omega\Subset K\Subset U\). Define
\[
    G(z):=\frac12|z|^2-F(z),\qquad
    G_K^*(y):=\sup_{z\in K}\{\langle z,y\rangle-G(z)\},
    \qquad
    V(y):=G_K^*(y)-\frac12|y|^2.
\]
Here $G_K^*$ is the restricted Legendre transform of $G$, and \eqref{eq:F-curvature} gives $D^2G\geq \kappa_F I$.

Since $G$ is convex and \(Z_x=\sum_{i=1}^N\lambda_i x_i\in\Omega\Subset K\), the restricted transform satisfies
\begin{equation}
	G(Z_x)=\sup_{y\in\mathbb R^n}
	\{\langle Z_x,y\rangle-G_K^*(y)\}.
\end{equation}
Then we have
\begin{equation}
\begin{aligned}
		c_F(x_1,\ldots,x_N)&=\sum_{i=1}^N\frac{\lambda_i}{2}|x_i|^2
		-\sup_{y\in\mathbb R^n}
		\{\langle Z_x,y\rangle-G_K^*(y)\}\\
		&=\inf_{y\in\mathbb R^n}
		\left\{
		\sum_{i=1}^N\frac{\lambda_i}{2}|x_i-y|^2+V(y)
		\right\}.
\end{aligned}
\end{equation}
	Moreover, by the Fenchel--Young equality, the unique minimizer is $\nabla G(Z_x)$.

We therefore define
\[
    H^F:=\nabla G(Z^F)=Z^F-\nabla F(Z^F),
    \quad
    \mu_{\mathbb P}^F:=H^F_\#\rho.
\]
The penalized barycenter formulation can be verified directly from the multi-marginal dual potentials. Write $X_1=S$, $X_i=Y$ for $i\geq2$, and define
\[
 f_i(x_i):=\frac12|x_i|^2-\frac{\varphi_i^F(x_i)}{\lambda_i},\qquad
 f_i^*(y):=\sup_{x_i\in X_i}\{\langle x_i,y\rangle-f_i(x_i)\},\qquad
 \psi_{\mu_{\mathbb P}}^{F,i}(y):=\frac12|y|^2-f_i^*(y).
\]
Dual feasibility gives $\sum_i\lambda_i f_i(x_i)\geq G(Z_x)$. Since the coordinates can be varied independently, taking suprema yields
\begin{equation}\label{pen}
 \sum_{i=1}^N\lambda_i\psi_{\mu_{\mathbb P}}^{F,i}(y)+V(y)
 =G_K^*(y)-\sum_{i=1}^N\lambda_i f_i^*(y)\geq0.
\end{equation}
At a contact tuple $t=(x_1,T_2^F(x_1),\dots,T_N^F(x_1))$, set $y=\nabla G(Z^F(x_1))$. Changing only coordinate $i$ and using convexity of $G$ gives
\[
 f_i(x_i')-f_i(t_i)\geq
 \lambda_i^{-1}\bigl[G(Z^F+\lambda_i(x_i'-t_i))-G(Z^F)\bigr]
 \geq\langle y,x_i'-t_i\rangle.
\]
It follows that $t_i$ attains the supremum defining $f_i^*(y)$ for every $i$. Consequently, equality holds in \eqref{pen} for $\mu_{\mathbb P}^F$-a.e. $y$, and the pairs $(t_i,y)$ satisfy quadratic Kantorovich equality. These equalities identify $H^F$ as the optimal map from $\rho$ to $\mu_{\mathbb P}^F$ and $\psi_{\mu_{\mathbb P}}^{F,i}$ as a Kantorovich potential from $\mu_{\mathbb P}^F$ to $\mu_i$. Integrating \eqref{pen} against an arbitrary competing measure, and using the equality just proved, shows that $\mu_{\mathbb P}^F$ minimizes
\[
 \mathcal{F}_{\mathbb P}^V(\nu):=\mathcal{F}_{\mathbb P}(\nu)+\int V\,\d\nu.
\]
This is a penalized Wasserstein barycenter problem, as studied in \cite{BigotCazellesPapadakis2019}. The same construction applies to $\tilde{\mathbb P}$ and $\tilde H^F$.

The affine normalization translates $H^F$ and $\tilde H^F$ by the same vector. Since $\nabla F(z_0)=0$ and $\|D^2G\|\leq1+\tau_F$, both transformed measures are supported in a common compact set $\Omega_F$ such that $\diam(\Omega_F\cup\Omega)$ is controlled by $D$ and $\tau_F$.

Using the penalized balance \eqref{pen} in the same way as \eqref{eq:barycenter-potential-balance}, we obtain
\begin{align*}
 \mathcal{F}_{\mathbb P}^V(\mu_{\tilde{\mathbb P}}^F)
 -\mathcal{F}_{\mathbb P}^V(\mu_{\mathbb P}^F)
 &\geq\lambda_1D_\rho(\mu_{\tilde{\mathbb P}}^F,\mu_{\mathbb P}^F),\\
 \mathcal{F}_{\tilde{\mathbb P}}^V(\mu_{\mathbb P}^F)
 -\mathcal{F}_{\tilde{\mathbb P}}^V(\mu_{\tilde{\mathbb P}}^F)
 &\geq\lambda_1D_\rho(\mu_{\mathbb P}^F,\mu_{\tilde{\mathbb P}}^F),
\end{align*}
where the defects are defined as in \eqref{2.30} and \eqref{2.31}. Adding these inequalities cancels the two penalty contributions and yields
\begin{equation}\label{D}
 \begin{aligned}
 \mathcal{F}_{\mathbb P}(\mu_{\tilde{\mathbb P}}^F)-\mathcal{F}_{\mathbb P}(\mu_{\mathbb P}^F)+\mathcal{F}_{\tilde{\mathbb P}}(\mu_{\mathbb P}^F)-\mathcal{F}_{\tilde{\mathbb P}}(\mu_{\tilde{\mathbb P}}^F)\geq\lambda_1\left[D_\rho(\mu_{\tilde{\mathbb P}}^F,\mu_{\mathbb P}^F)
                         +D_\rho(\mu_{\mathbb P}^F,\mu_{\tilde{\mathbb P}}^F)\right].
 \end{aligned}
\end{equation}

By \eqref{eq:F-curvature},
\begin{equation}\label{E}
	\|Z^F-\tilde{Z}^F\|_{L^2(\rho)}\leq \kappa_F^{-1}\|H^F-\tilde{H}^F\|_{L^2(\rho)}.
\end{equation}

Combining \eqref{D}, \eqref{E} and Proposition \ref{prop:collective-coercivity} with Proposition~\ref{prop:two-marginal-defect}, and repeating the proof of Theorem \ref{multi-marginal}, completes the proof.
\end{proof}

\paragraph{Uniformly concave costs of the sum.}\label{subsec:sum-costs}
Consider costs of the sum, in the form studied by
Gangbo--\'Swi\c{e}ch and Heinich \cite{zbMATH01060715,heinich2002monge}:
\begin{equation}\label{eq:sum-cost}
    c_h(x_1,\dots,x_N)=h\left(\sum_{i=1}^N\lambda_i x_i\right),
\end{equation}
where \(h\in C^2(U)\) for an open convex neighborhood \(U\supseteq\Omega\). Since
adding functions of the individual variables does not change the optimal transport
plans, the cost \eqref{eq:sum-cost} is equivalent to the cost function
\begin{equation}\label{eq:sum-to-collective}
    c_0(x_1,\dots,x_N)+F_h(Z_x),
    \quad F_h(z):=h(z)+\frac12|z|^2.
\end{equation}
The quadratic cost \eqref{cost} corresponds to \(h(z)=-\frac12|z|^2\), and
then \(F_h=0\). The following result shows that the uniformly concave costs of the sum are covered by
Proposition \ref{prop:collective-full-stability}.

\begin{corollary}\label{cor:sum-costs}
Assume that \(h\in C^2(U)\), with \(U\supseteq\Omega\), satisfies, as quadratic forms on \(U\),
\begin{equation}\label{eq:h-uniform-concavity}
    -\tau_h I\leq D^2h\leq -\kappa_h I,
\end{equation}
for some constants \(0<\kappa_h\le \tau_h<\infty\). Let $\rho\in\mathcal P_{\mathrm{ac}}(S)\cap L^\infty(S)$ and let \(c_h\) be defined by \eqref{eq:sum-cost}. Then there exists a constant
\(C>0\), depending only on $\lambda_1, N, \diam(S\cup Y),\kappa_h, \tau_h, n$, and $\|\rho\|_{L^\infty}$, such that
\begin{equation*}
    \sum_{i=2}^N \lambda_i
    \|T^h_i-\tilde T^h_i\|_{L^2(\rho)}
    \leq C\left(\sum_{i=2}^N \lambda_iW_1(\mu_i,\tilde\mu_i)\right)^{1/4},
\end{equation*}
where
\((T_2^h,\dots,T_N^h)\) and
\((\tilde T_2^h,\dots,\tilde T_N^h)\) are the unique multi-marginal optimal transport maps associated with $(\rho,\mu_2,\dots,\mu_N)$ and $(\rho,\tilde\mu_2,\dots,\tilde\mu_N)$, respectively.
\end{corollary}

\begin{proof}
	The existence and uniqueness of multi-marginal optimal transport maps follow from \cite[Theorem 3.1]{kim2014general}.

	Denote \(\kappa_*:=\min\{\kappa_h,1\}\), and \(\tau_*:=\max\{\tau_h-1,0\}\). Then by \eqref{eq:sum-to-collective} and \eqref{eq:h-uniform-concavity},
\begin{equation*}
	  D^2F_h\ge (1-\tau_h)I\ge -\tau_*I,
	\quad
	D^2F_h\le (1-\kappa_h)I\le (1-\kappa_*)I.
\end{equation*}
Thus \eqref{eq:F-curvature} holds with \(\kappa_F=\kappa_*\) and \(\tau_F=\tau_*\), and Proposition \ref{prop:collective-full-stability} proves the corollary.
\end{proof}

\subsection{Pairwise quadratic costs}\label{subsec:pairwise-quadratic}
We next consider pairwise quadratic interactions, for which qualitative Monge and uniqueness results for graph-based costs were obtained in \cite{PassVargasJimenez2023}. Our aim is to estimate the distance from an arbitrary comparison plan to the optimal transport maps while keeping the source marginal fixed. An inertia condition on the interaction matrix allows us to control all target coordinates from the collective coordinate and the Kantorovich defect. As before, coordinate $1$ is the source $\mu_1=\rho$, and $I=\{2,\dots,N\}$ denotes the target indices.

Consider
\begin{equation}\label{eq:pairwise-cost}
 c_A(x_1,\dots,x_N):=\frac12\sum_{1\leq i<j\leq N}a_{ij}|x_i-x_j|^2,
 \qquad A=A^{\mathsf T},\quad a_{ii}=0,\quad a_{ij}\geq0.
\end{equation}
The nonnegativity assumption is entrywise. Write
\[
 I:=\{2,\dots,N\},\qquad a:=(a_{12},\dots,a_{1N})^{\mathsf T},
 \qquad A_I:=(a_{ij})_{i,j\in I},\qquad d_i:=\sum_{j\neq i}a_{ij}.
\]
The source derivative is
\[
 D_{x_1}c_A=d_1x_1-Z_x^A,\qquad Z_x^A:=\sum_{i=2}^Na_{1i}x_i.
\]
The source derivative therefore depends on the target tuple only through $Z_x^A$, and relative displacements satisfying $\sum_{i=2}^Na_{1i}\Delta x_i=0$ leave it unchanged. To control these displacements, we need positivity of $-A_I$ on this subspace. The defect identity produces this form from the mixed interaction terms; the graph Laplacian of the cost alone does not provide the required coercivity.

\begin{lemma}\label{lem:quadratic-inertia}
Assume $a\neq0$. The following conditions are equivalent:
\begin{itemize}
 \item[\textup{(a)}] $A$ has one positive and $N-1$ negative eigenvalues;
 \item[\textup{(b)}] $-A_I$ is positive definite on $a^\perp$;
 \item[\textup{(c)}] there exist $\kappa_A>0$ and $C_A\geq0$ such that
 \begin{equation}\label{eq:anchored-quadratic-bound}
 -A_I\geq\kappa_A\Id-C_Aaa^{\mathsf T}.
 \end{equation}
\end{itemize}
\end{lemma}
\begin{proof}
Choose an orthonormal basis of the target coordinate space beginning with $a/|a|$. In this basis $A$ has the form
\[
 \begin{pmatrix}
 0&|a|&0\\ |a|&b&v^{\mathsf T}\\0&v&C
 \end{pmatrix},\qquad C=A_I|_{a^\perp}.
\]
A congruence transformation that changes the source coordinate eliminates the block $v$. The remaining $2\times2$ block has determinant $-|a|^2$, so it contributes one positive and one negative eigenvalue; Sylvester's law of inertia then proves the equivalence of \textup{(a)} and \textup{(b)}. To deduce \textup{(c)} from \textup{(b)}, decompose a target vector into its components along $a$ and $a^\perp$. Positivity on the latter subspace, together with Young's inequality for the mixed term, gives \eqref{eq:anchored-quadratic-bound}. Conversely, restricting \eqref{eq:anchored-quadratic-bound} to $a^\perp$ gives \textup{(b)}.
\end{proof}

For example, the uniformly weighted complete graph has $A=J-I_N$, where $J$ is the all-ones matrix. Its eigenvalues are $N-1$ and $-1$ (with multiplicity $N-1$), so it satisfies the inertia condition. A disconnected nonnegative interaction graph cannot satisfy it: an isolated vertex gives a zero eigenvalue, while two components containing edges give at least two positive eigenvalues. Connectedness alone is not sufficient, as the star cost discussed at the end of this subsection shows.

To estimate the collective coordinate, we use an affine defect bound derived from Ford's subgradient-oscillation estimate \cite[Theorem~2.2, $q=1$]{Ford2026} and pointwise Fenchel estimate \cite[Lemma~2.4]{Ford2026}. Theorem~\ref{thm:quadratic-source-fixed} will combine this bound with Lemma~\ref{lem:quadratic-inertia} to control the individual coordinates. The argument does not require a bounded source density and also applies to singular measures satisfying a suitable growth condition. More precisely, throughout the rest of this subsection we assume that, for some $M_\rho>0$ and $0<\theta\leq1$,
\begin{equation}\label{eq:source-growth}
 \rho(B_r(x))\leq M_\rho r^{n-1+\theta},\qquad x\in\R^n,\quad r>0.
\end{equation}
A bounded density satisfies \eqref{eq:source-growth} with $\theta=1$ and $M_\rho=\mathcal L^n(B_1(0))\|\rho\|_{L^\infty}$. Conversely, at $\theta=1$ the ball bound implies a bounded Lebesgue density, up to a dimensional constant.

For $0<\theta<1$, the condition permits singular measures, but their mass must be spread over sets of dimension strictly greater than $n-1$: every set of Hausdorff dimension less than $n-1+\theta$ has zero $\rho$-mass. This excludes measures carried by smooth hypersurfaces, including measures supported on curves in dimension two. A concrete singular source satisfying the growth bound is
\[
 \rho=\bigl(\mathcal L^{n-1}\restr{[0,1]^{n-1}}\bigr)\otimes\nu_C,
 \qquad \theta=\frac{\log2}{\log3},
\]
where $\nu_C$ is the standard middle-thirds Cantor probability measure. Its interval bound $\nu_C([t-r,t+r])\leq C r^\theta$, combined with the volume bound in the other factors, gives \eqref{eq:source-growth} by projecting a ball onto those factors. In dimension $n=1$, the example reduces to $\rho=\nu_C$.

\begin{lemma}\label{lem:affine-defect}
Let $S\subset\R^n$ be compact and let $\rho\in\mathcal P(S)$ satisfy \eqref{eq:source-growth}. Let $\mathcal Y$ be a compact metric space, and let $v:\mathcal Y\to\R^n$ and $q:\mathcal Y\to\R$ be continuous. Define
\[
 u(x):=\max_{y\in\mathcal Y}\{\langle x,v(y)\rangle-q(y)\},\qquad
 r(x,y):=u(x)+q(y)-\langle x,v(y)\rangle.
\]
Then $u$ is differentiable $\rho$-a.e., and every probability measure $\Gamma$ on $S\times\mathcal Y$ with first marginal $\rho$ satisfies
\begin{equation}\label{eq:affine-defect-estimate}
 \int|v(y)-\nabla u(x)|^2\,\d\Gamma(x,y)
 \leq C\left(\int r(x,y)\,\d\Gamma(x,y)\right)^{\theta/(1+\theta)}.
\end{equation}
The constant depends only on $n,\theta,M_\rho,\diam(S)$ and $\diam(v(\mathcal Y))$.
\end{lemma}
\begin{proof}
Write $L=\diam(v(\mathcal Y))$ and $\omega_s(x)=\diam\partial u(\overline B_s(x))$. When $L=0$, all slopes coincide and the assertion follows immediately. We may therefore assume that this diameter is positive. For $n\geq2$, Ford's oscillation estimate \cite[Theorem~2.2, $q=1$]{Ford2026} gives
\[
 \int\omega_s\,\d\rho\leq C_0s^\theta,
 \qquad C_0=c_nM_\rho\diam(S)^{n-1}L.
\]
Here one translates $S$ into a ball of radius $\diam(S)$ and uses that the entire subgradient image of $u$ lies in $\operatorname{conv}(v(\mathcal Y))$.
For $n=1$, the same bound follows from Fubini: the positive measure $u''$ has total mass at most $L$, and
\[
 \int u''([x-s,x+s])\,\d\rho(x)\leq M_\rho(2s)^\theta L.
\]
Letting $s\downarrow0$ in the oscillation bound shows that the subdifferential is a singleton at $\rho$-a.e. point, and hence that the convex potential is differentiable there. At each such point, the pointwise Fenchel estimate \cite[Lemma~2.4]{Ford2026}, together with $u^*(v(y))\leq q(y)$, gives
\[
 |v(y)-\nabla u(x)|\leq\omega_s(x)+\frac{r(x,y)}{s}.
\]
Both slopes belong to $\operatorname{conv}(v(\mathcal Y))$, so their distance is bounded by the diameter of this set. Multiplying by that distance and integrating with respect to the comparison plan therefore yields
\[
 \int|v(y)-\nabla u(x)|^2\,\d\Gamma
 \leq LC_0s^\theta+\frac{L}{s}\int r\,\d\Gamma.
\]
Optimizing the right-hand side over $s>0$ gives the claimed estimate; if the integrated defect is zero, the same conclusion follows by letting $s\downarrow0$. To track the constant when $\theta=1$, write $J=\int r\,\d\Gamma$. For a positive defect, the choice $s=(J/C_0)^{1/2}$ gives
\[
 \left(\int|v(y)-\nabla u(x)|^2\,\d\Gamma\right)^2
 \leq4L^2C_0J
 \leq C_nM_\rho\diam(S)^{n-1}L^3J.
\]
The same formula holds in dimension one with $\diam(S)^0=1$. Thus the coefficient in this fourth-power estimate grows linearly in $M_\rho$.
\end{proof}

\begin{theorem}[Coercivity with a fixed source]\label{thm:quadratic-source-fixed}
Let $N\geq3$, let $S,Y\subset\R^n$ be compact, and let $\rho\in\mathcal P(S)$ satisfy \eqref{eq:source-growth}. Assume that the matrix $A$ in \eqref{eq:pairwise-cost} has one positive and $N-1$ negative eigenvalues. Then for every $\mu_2,\dots,\mu_N\in\mathcal P(Y)$, the $c_A$-optimal transport plan is unique and has the form
\[
 \gamma_A=(\Id,T_2^A,\dots,T_N^A)_\#\rho.
\]
Choose optimal $c_A$-conjugate potentials $(\varphi_1^A,\dots,\varphi_N^A)$ and set
\[
 R_A(x_1,\dots,x_N):=c_A(x_1,\dots,x_N)-\sum_{i=1}^N\varphi_i^A(x_i),
 \qquad Z^A(x_1):=\sum_{i=2}^Na_{1i}T_i^A(x_1).
\]
Every probability measure $\Gamma$ on $S\times Y^{N-1}$ with first marginal $\rho$ satisfies
\begin{equation}\label{eq:quadratic-source-fixed}
 \int\sum_{i=2}^N|x_i-T_i^A(x_1)|^2\,\d\Gamma
 \leq C\left(\int R_A\,\d\Gamma\right)^{\theta/(1+\theta)},
\end{equation}
where $C$ depends only on $A,n,\theta,M_\rho$ and $\diam(S\cup Y)$. If $Z_x^A=Z^A(x_1)$ $\Gamma$-a.e., then
\begin{equation}\label{eq:quadratic-fixed-collective}
 \int\sum_{i=2}^N|x_i-T_i^A(x_1)|^2\,\d\Gamma
 \leq\frac{2}{\kappa_A}\int R_A\,\d\Gamma,
\end{equation}
with $\kappa_A$ as in \eqref{eq:anchored-quadratic-bound}.
\end{theorem}
\begin{proof}
Kantorovich duality on compact spaces provides continuous conjugate optimal potentials. The conjugacy relation for the source potential gives a convex function on all of $\R^n$:
\begin{equation}\label{eq:quadratic-source-potential}
 \begin{split}
 u(x_1)&:=\max_{(x_i)_{i\in I}\in Y^{N-1}}
 \left\{\left\langle x_1,\sum_{i=2}^Na_{1i}x_i\right\rangle-q_A(x_2,\dots,x_N)\right\},\\
 q_A(x_2,\dots,x_N)&:=\frac12\sum_{i=2}^Nd_i|x_i|^2
 -\sum_{2\leq i<j\leq N}a_{ij}\langle x_i,x_j\rangle
 -\sum_{i=2}^N\varphi_i^A(x_i).
 \end{split}
\end{equation}
On $S$, the conjugacy relation gives $u(x_1)=\frac{d_1}{2}|x_1|^2-\varphi_1^A(x_1)$, so the affine defect associated with this representation is exactly $R_A$. Moreover, the slope image has diameter at most $d_1\diam(Y)$, which gives the geometric bound required to apply the preceding lemma.

For each $x_1\in S$, continuity of the objective and compactness of the target space ensure that the maximum in \eqref{eq:quadratic-source-potential} is attained. The measurable maximum theorem therefore provides a Borel choice of contact tuple $(T_i^A(x_1))_{i\in I}$. Lemma~\ref{lem:affine-defect} identifies its collective coordinate as $Z^A(x_1)=\nabla u(x_1)$ $\rho$-a.e. and gives
\begin{equation}\label{eq:quadratic-collective-control}
 \int|Z_x^A-Z^A(x_1)|^2\,\d\Gamma
 \leq C_1\left(\int R_A\,\d\Gamma\right)^{\theta/(1+\theta)}.
\end{equation}
At every differentiability point of $u$, all contact tuples have the same collective coordinate $\nabla u(x_1)$. We now use pointwise coercivity to show that their individual coordinates also agree. Once this uniqueness is established, applying the integrated estimate to an optimal transport plan will verify that the selected maps have the prescribed marginals. The choice on the $\rho$-null set of nondifferentiability does not affect either conclusion.

Fix a source point and its selected contact tuple, and put $\Delta x_i=x_i-T_i^A(x_1)$. As in Proposition~\ref{prop:kantorovich-defect-lower}, let $r_i^A$ be the defect after changing only coordinate $i$. The exact expansion is
\begin{equation}\label{eq:quadratic-exact-defect}
 2R_A=2\sum_{i=2}^Nr_i^A
       -\sum_{i,j\in I}a_{ij}\langle\Delta x_i,\Delta x_j\rangle,
 \qquad r_i^A\geq0.
\end{equation}
Applying \eqref{eq:anchored-quadratic-bound} in each spatial coordinate and using the nonnegativity of the single-coordinate defects, we obtain
\[
 \kappa_A\sum_{i=2}^N|\Delta x_i|^2
 \leq2R_A+C_A|Z_x^A-Z^A(x_1)|^2.
\]
In particular, if $(x_2,\dots,x_N)$ is another contact tuple at a differentiability point of $u$, then $R_A=0$ and $Z_x^A=Z^A(x_1)=\nabla u(x_1)$, so the preceding inequality forces $x_i=T_i^A(x_1)$ for every $i$. Thus the contact tuple is unique at $\rho$-a.e. source point.

Writing $\mathcal D_A=\int R_A\,\d\Gamma$ and integrating, we obtain
\[
 \kappa_A\int\sum_{i=2}^N|\Delta x_i|^2\,\d\Gamma
 \leq2\mathcal D_A+C_AC_1\mathcal D_A^{\theta/(1+\theta)}.
\]
For $\mathcal D_A\leq1$, the preceding inequality proves \eqref{eq:quadratic-source-fixed}. When $\mathcal D_A\geq1$, the uniform bound $(N-1)\diam(Y)^2$ on the left-hand integral gives the same conclusion after increasing the constant. If the collective coordinates agree, their contribution vanishes before integration, and the estimate reduces directly to \eqref{eq:quadratic-fixed-collective}.

Finally, every optimal transport plan has zero integrated defect, so \eqref{eq:quadratic-source-fixed} forces it to coincide with $(\Id,T_2^A,\dots,T_N^A)_\#\rho$. Since an optimal transport plan exists, the selected maps have the prescribed marginals; the same argument also proves uniqueness of the plan and of the maps almost everywhere with respect to the source measure.
\end{proof}

With the source marginal fixed, the preceding estimate also compares the reference solution with plans for perturbed target marginals and a perturbed cost. The following corollary allows such a comparison plan to be only approximately optimal.

\begin{corollary}\label{cor:quadratic-approximate}
Under the assumptions of Theorem~\ref{thm:quadratic-source-fixed}, let $\tilde\mu_i\in\mathcal P(Y)$, and let $\tilde c$ be continuous on $S\times Y^{N-1}$. Suppose $\Gamma\in\Pi(\rho,\tilde\mu_2,\dots,\tilde\mu_N)$ satisfies
\[
 \int\tilde c\,\d\Gamma
 \leq\inf_{\pi\in\Pi(\rho,\tilde\mu_2,\dots,\tilde\mu_N)}\int\tilde c\,\d\pi+\epsilon,
 \qquad\epsilon\geq0.
\]
Set $D=\diam(S\cup Y)$, $\eta=\|\tilde c-c_A\|_{L^\infty(S\times Y^{N-1})}$ and
$\delta_A=\sum_{i=2}^Nd_iW_1(\mu_i,\tilde\mu_i)$. Then
\begin{equation}\label{eq:quadratic-approximate}
 \left(\int\sum_{i=2}^N|x_i-T_i^A(x_1)|^2\,\d\Gamma\right)^{1/2}
 \leq C\bigl(\epsilon+2\eta+2D\delta_A\bigr)^{\theta/(2(1+\theta))}.
\end{equation}
In particular, for $\tilde c=c_A$ and its optimal transport maps $\tilde T_i^A$,
\begin{equation}\label{eq:quadratic-map-stability}
 \left(\sum_{i=2}^N\|T_i^A-\tilde T_i^A\|_{L^2(\rho)}^2\right)^{1/2}
 \leq C\left(\sum_{i=2}^Nd_iW_1(\mu_i,\tilde\mu_i)\right)^{\theta/(2(1+\theta))}.
\end{equation}
The constants have the same dependence as in Theorem~\ref{thm:quadratic-source-fixed}. For bounded source densities the exponent is $\frac14$.
\end{corollary}
\begin{proof}
Let $\mathcal C_A(\mu_2,\dots,\mu_N)$ denote the optimal value for $c_A$ with first marginal $\rho$. Since $|D_{x_i}c_A|\leq d_iD$ on the convex hull of the supports, conjugacy gives $\Lip(\varphi_i^A|_Y)\leq d_iD$. The gluing argument of Lemma~\ref{cost-difference} also gives
\[
 |\mathcal C_A(\tilde\mu_2,\dots,\tilde\mu_N)-\mathcal C_A(\mu_2,\dots,\mu_N)|
 \leq D\delta_A.
\]
Approximate optimality and the cost bound imply
\[
 \int c_A\,\d\Gamma\leq\mathcal C_A(\tilde\mu_2,\dots,\tilde\mu_N)+\epsilon+2\eta.
\]
Consequently, dual optimality for the reference marginals yields
\begin{align*}
 \int R_A\,\d\Gamma
=\int c_A\,\d\Gamma-\mathcal C_A(\mu_2,\dots,\mu_N)
       +\sum_{i=2}^N\int\varphi_i^A\,\d(\mu_i-\tilde\mu_i)\leq\epsilon+2\eta+2D\delta_A.
\end{align*}
Applying Theorem~\ref{thm:quadratic-source-fixed} and taking square roots proves the estimate for comparison plans. The map estimate \eqref{eq:quadratic-map-stability} follows by choosing $\Gamma=(\Id,\tilde T_2^A,\dots,\tilde T_N^A)_\#\rho$.
\end{proof}

\begin{example}\label{ex:nonfactorized-quadratic}
For $N=4$, take $a_{12}=a_{21}=t$ and all other off-diagonal entries equal to $1$. The eigenvalues of $A$ are
\[
 -t,\quad -1,\quad\frac{t+1+\sqrt{(t-1)^2+16}}2,
 \quad\frac{t+1-\sqrt{(t-1)^2+16}}2.
\]
Thus Theorem~\ref{thm:quadratic-source-fixed} applies for $0<t<4$. For $t\neq1$, the coefficients cannot be written as $a_{ij}=m_im_j$: such a factorization would require $a_{12}a_{34}=a_{13}a_{24}$, hence $t=1$. The theorem therefore covers interactions beyond the factorized barycentric cost. Its constants may deteriorate as $t$ approaches $0$ or $4$, where the strict inertia condition is lost.
\end{example}

\begin{remark}
The inertia condition is sufficient to control all target coordinates using the collective coordinate $Z_x^A$ and the quadratic form in \eqref{eq:quadratic-exact-defect}. It is not necessary for stability of a multi-marginal problem. For example, a star cost with $a_{1i}>0$ and $a_{ij}=0$ for $i,j\geq2$ splits into independent two-marginal quadratic problems, although $A_I=0$. The optimal exponent $\frac14$ proved earlier concerns the barycentric class under general bounded-density assumptions; no sharpness assertion is made here for each individual matrix $A$.
\end{remark}

\bigskip

\appendix
\section{A fixed-source construction}\label{app:fixed-source-construction}
This appendix gives the geometric construction used in \eqref{eq:weight-scales}. The source is fixed throughout and is supported on separated cells. Within cell $k$, the two target pairs can be matched horizontally over a distance $2a_ks_k$, whereas the optimal images of the same source point are vertically separated by $2s_k$; see Figure~\ref{fig:cell-geometry}. Because the active cell has mass of order $a_ks_k^n$, these two distances produce an $L^2$ map discrepancy of order $a_k^{1/2}s_k^{(n+2)/2}$ and a $W_1$ cost of order $a_k^2s_k^{n+1}$. Choosing $a_k\asymp\lambda_k$, with $\lambda_k$ decaying faster than every power of $s_k$, then gives the required failure of uniform estimates for one fixed source.

\begin{figure}[hbp]
\centering
\begin{tikzpicture}[x=1.05cm,y=1.2cm,>=stealth]
 \filldraw[fill=black!7,draw=black!55] (-5,-.5) rectangle (-3,.5);
 \filldraw[fill=black!7,draw=black!55] (3,-.5) rectangle (5,.5);
 \node at (-4,.82) {$V_-^{a_k}$};
 \node at (4,.82) {$V_+^{a_k}$};
 \draw[<->,dashed] (-.5,1) -- (.5,1);
 \draw[<->,dashed] (-.5,-1) -- (.5,-1);
 \draw[<->] (-.5,-.94) -- (-.5,.94);
 \draw[<->] (.5,-.94) -- (.5,.94);
 \fill (-.5,1) circle (1.7pt) node[above left] {$C_k$};
 \fill (.5,1) circle (1.7pt) node[above right] {$A_k$};
 \fill (-.5,-1) circle (1.7pt) node[below left] {$B_k$};
 \fill (.5,-1) circle (1.7pt) node[below right] {$D_k$};
 \fill (0,0) circle (1.7pt) node[below] {$E_k$};
 \node[left] at (-.55,0) {$2s_k$};
 \node[above] at (0,1.38) {$2a_ks_k$};
 \node at (-4,-.82) {$(T_2^k,T_3^k)=(B_k,C_k)$};
 \node at (4,-.82) {$(T_2^k,T_3^k)=(A_k,D_k)$};
\end{tikzpicture}
\caption{The active cell in dimension two, after translation by $-p_k$ and rescaling by $s_k^{-1}$ (schematic). Dashed horizontal arrows indicate a matching of the two target measures with small transport cost; solid vertical segments join the optimal images of a common source point. Distance labels are in the original scale.}
\label{fig:cell-geometry}
\end{figure}

\begin{samepage}
\begin{lemma}\label{lem:weight-construction}
For every $n\geq2$ and $N\geq3$, there exist compact sets $S,Y\subset\R^n$, a bounded-density source $\rho\in\mathcal P_{\mathrm{ac}}(S)$, an integer $k_0\geq0$ and $\eta>0$ with the following property. For
\[
 \lambda_k=2^{-(k+k_0)^2},\quad s_k=\eta2^{-k},\quad
 \lambda_1^k=1-(N-1)\lambda_k\geq\tfrac12,\quad
 \lambda_2^k=\cdots=\lambda_N^k=\lambda_k,
\]
there are target tuples $(\mu_2^k,\dots,\mu_N^k)$ supported in $Y$ whose multi-marginal optimal transport maps satisfy \eqref{eq:weight-scales}, with comparison constants independent of $k$.
\end{lemma}
\end{samepage}
\begin{proof}
\textbf{Step 1:}
			Fix $n\geq 2$ and $N\geq 3$. Choose an integer $k_0\geq0$, depending only on $N$, such that $2^{-(1+k_0)^2}\leq\frac{1}{2(N-1)}$, and set
            \[
            \lambda_k:=2^{-(k+k_0)^2},\qquad k\geq1.
            \]
            Then $0<\lambda_k\leq\frac{1}{2(N-1)}$ for every $k\geq1$. Denote
			\begin{equation*}
				\lambda_{1}^k:=1-(N-1)\lambda_{k}\geq \frac{1}{2},\quad \lambda_{2}^k=\dots=\lambda_{N}^k:=\lambda_k.
			\end{equation*}
			Denote $a_k:=\frac{\lambda_k}{\lambda_{1}^k}\leq 1$ and $K_{n-2}:=[-\frac{1}{2}, \frac{1}{2}]^{n-2}$ with the convention that $K_0$ is a single point. Define two subsets of $\R^n$ by
			\begin{equation*}
				V_+^{a_k}:=[3,5]\times [-a_k,a_k]\times K_{n-2}, \quad V_-^{a_k}:= [-5,-3]\times [-a_k,a_k]\times K_{n-2}.
			\end{equation*}

			Denote $p_k:=2^{-k}e_1$ and $s_k:=\eta2^{-k}$, where $\eta\leq 1$ depends on $n$ and $N$ and is chosen in \eqref{eta}
			below. Define the $k$th cell $Z_k$ by
			\begin{equation*}
				Z_k:=p_k+s_k(V_+^{a_k}\cup V_-^{a_k}), \quad k\in \N,
			\end{equation*}
			and define
			\begin{equation*}
			V_{k,+}:= p_k+s_kV_+^{a_k},\quad\text{and} \quad V_{k,-}:= p_k+s_kV_-^{a_k}.
			\end{equation*}
				Define the source set $S$ and the source measure $\rho$ by
			\begin{equation*}
				S:=\overline{\bigcup_{k=1}^\infty Z_k}=\bigcup_{k=1}^\infty Z_k\cup \{0\}\subseteq B(0,R), \quad \rho:=\frac{\mathcal{L}^n\restr{S}}{\mathcal{L}^n(S)}\in\mathcal P_{\mathrm{ac}}(S),
			\end{equation*}
			where $R$ depends only on $n$. By the estimate \eqref{lk} below, the sets $Z_k$ are pairwise disjoint, and hence
			\begin{equation}
				\mathcal{L}^n(S)=\sum_{k=1}^{\infty}\mathcal{L}^n(Z_k)=\sum_{k=1}^{\infty}8a_ks_k^n<+\infty.
			\end{equation}
			Thus $\rho\in\mathcal P_{\mathrm{ac}}(S)\cap L^\infty(S)$.

			For $k\in\N$, define five points in $\R^n$ by
			\begin{align*}
				A_k&:=p_k+s_k(a_k,1,0,\dots,0), &
				B_k&:=p_k+s_k(-a_k,-1,0,\dots,0),\\
				C_k&:=p_k+s_k(-a_k,1,0,\dots,0), &
				D_k&:=p_k+s_k(a_k,-1,0,\dots,0),
			\end{align*}
			and $E_k:=p_k+s_k(0,0,0,\dots,0)$. Define the target set by
			\begin{equation*}
				Y_k:={\{A_k,B_k,C_k,D_k,E_k\}}, \quad k\in\N, \quad Y:=\overline{\bigcup_{k=1}^\infty Y_k}=\bigcup_{k=1}^\infty Y_k\cup \{0\}\subseteq B(0, R).
			\end{equation*}
			For $k\in\N$, denote
			\begin{equation*}
				v_k:=\rho(V_{k,+})=\rho(V_{k,-})=\frac{4a_ks_k^n}{\mathcal{L}^n(S)}.
			\end{equation*}
		     Define the target measures $(\mu_2^k, \dots, \mu_N^k)$ by
		     \begin{equation*}
		     	\mu_{2}^k:=v_k\delta_{A_k}+v_k\delta_{B_k}+\sum_{l\ne k}\rho(Z_l)\delta_{E_l}, \quad  \mu_{3}^k:=v_k\delta_{D_k}+v_k\delta_{C_k}+\sum_{l\ne k}\rho(Z_l)\delta_{E_l},
		     \end{equation*}
			and
			\begin{equation*}
				\mu_{4}^k=\dots=\mu_{N}^k:=\sum_{l=1}^\infty\rho(Z_l)\delta_{E_l}.
			\end{equation*}

			By \cite{zbMATH01060715} (see also \cite{zbMATH06484002}), there exists a unique tuple of multi-marginal optimal transport maps $(T_2^k, T_3^k, \dots, T_N^k)$. We claim that they are given by
			\begin{equation*}
				T_{2}^k(x_1)=
				 \begin{cases}
				 	A_k, & x_1\in V_{k,+},\\ B_k, & x_1\in V_{k,-},\\ E_l, & x_1\in Z_l,\ l\ne k,
				 \end{cases}
				 T_{3}^k(x_1)=
				 \begin{cases}
				 	D_k, & x_1\in V_{k,+},\\ C_k, & x_1\in V_{k,-},\\ E_l, & x_1\in Z_l,\ l\ne k,
				 	\end{cases}
			\end{equation*}
			and
			\begin{equation*}
				T_{4}^k(x_1)=\dots=T_{N}^k(x_1)= E_l, \quad x_1\in Z_l, \quad \rho\text{-a.e. } x_1\in S.
				\end{equation*}

				It is enough to show that these maps are pointwise optimal among all admissible target tuples.

			\textbf{Step 2:} Let $k\in\N$ and $x_1\in Z_k$. We claim that any admissible target tuple containing a point from $Y_l$, $l\neq k$, has strictly larger cost than any admissible target tuple contained in $Y_k^{N-1}$. Indeed, by the construction in Step 1, we know $Z_k, Y_k\subseteq B(p_k, Rs_k)$. Hence, for any $x_2,\dots,x_N\in Y_k$, by \eqref{equal}, we have
			\begin{equation}\label{L1}
				c(x_1, x_2, \dots ,x_N)\leq 2R^2\sum_{1\leq i<j\leq N}\lambda_i^k\lambda_j^k s_k^2\leq C_4R^2\lambda_k s_k^2,
			\end{equation}
			where $C_4$ depends only on $N$. We choose
			\begin{equation}\label{eta}
				\eta=\min\left\{1, \frac{1}{8(1+\sqrt{2C_4})R}\right\}>0.
			\end{equation}

             On the other hand, let $(y_2,\dots,y_N)$ be any target tuple with $y_i\in Y_l$ for $l\neq k$. Since $Z_k, Y_k\subseteq B(p_k, Rs_k)$ and $Z_l, Y_l\subseteq B(p_l, Rs_l)$, we have
             \[
             \dist(Z_k\cup Y_k, Z_l\cup Y_l)
             \ge |p_k-p_l|-R(s_k+s_l).
             \]
             We estimate this lower bound in two cases. If \(l>k\), then
             \[
             |p_k-p_l|\ge 2^{-k-1},
             \qquad
             s_k+s_l\le \frac32\eta 2^{-k}.
             \]
             Hence, by \eqref{eta},
             \[
             |p_k-p_l|-R(s_k+s_l)
             \ge
             \left(\frac12-\frac32R\eta\right)2^{-k}
             \ge 2\sqrt{2C_4}R\eta 2^{-k}
             =2\sqrt{2C_4}Rs_k.
             \]
             If \(l<k\), write \(m=k-l\ge1\). Then
             \[
             |p_k-p_l|=(2^m-1)2^{-k},
             \qquad
             s_k+s_l=(1+2^m)\eta2^{-k}.
             \]
             Again by \eqref{eta},
             \[
             |p_k-p_l|-R(s_k+s_l)
             \ge
             \left(1-3R\eta\right)2^{-k}
             \ge 2\sqrt{2C_4}R\eta2^{-k}
             =2\sqrt{2C_4}Rs_k.
             \]
             Therefore
             \begin{equation}\label{lk}
             	\dist(Z_k\cup Y_k, Z_l\cup Y_l)\geq |p_k-p_l|-R(s_k+s_l)\geq 2\sqrt{2C_4}Rs_k.
             \end{equation}
			Then by \eqref{equal}, we have
			\begin{equation}\label{L2}
				c(x_1, y_2,\dots,y_N)\geq \frac{\lambda_{1}^k\lambda_k}{2}|x_1-y_i|^2\geq \frac{\lambda_k}{4}\dist^2(Z_k,Y_l)\geq 2C_4R^2\lambda_k s_k^2.
			\end{equation}

		Combining \eqref{L1} and \eqref{L2} proves the claim. Hence, in the pointwise comparison on $Z_k$, it suffices
		to restrict to admissible target tuples contained in $Y_k^{N-1}$. We only need to consider tuples of the form
		\begin{equation*}
				(x_2,x_3,E_k,\dots,E_k), \quad x_2\in\{A_k,B_k\},\quad x_3\in\{C_k,D_k\}.
		\end{equation*}

		\textbf{Step 3:} We now compare these finitely many target tuples. First take $x_1\in V_{k,+}$. We claim that the unique pointwise minimizer is
		\begin{equation*}
			(A_k,D_k,E_k,\dots,E_k).
		\end{equation*}

          By \eqref{equal},
        \begin{equation}
        \begin{aligned}
        		&c(x_1,x_2, x_3, E_k,\dots, E_k)\\
        		=&\frac{1}{2}\lambda_1^k\lambda_k|x_1-x_2|^2+\frac{1}{2}\lambda_1^k\lambda_k|x_1-x_3|^2+\frac{1}{2}(\lambda_k)^2|x_2-x_3|^2\\
        		&+\frac{1}{2}(1-\lambda_{1}^k-2\lambda_k)\left(\lambda_1^k|x_1-E_k|^2+\lambda_k|x_2-E_k|^2+\lambda_k|x_3-E_k|^2\right).
        \end{aligned}
        \end{equation}

		Write $x_1:=p_k+s_k(u,v,z)\in V_{k,+}$ in coordinates $(u,v,z)\in\R\times\R\times\R^{n-2}$. Since
			\begin{equation}
				|A_k-E_k|=|B_k-E_k|=|C_k-E_k|=|D_k-E_k|,
			\end{equation}
			we have
			\begin{equation}
			\begin{aligned}
			c(x_1, A_k, C_k,\dots,E_k)-c(x_1, A_k, D_k,\dots,E_k)=2\lambda_ks_k^2\left(\lambda_1^k(ua_k-v)-\lambda_k(1-a_k^2)\right)>0,
			\end{aligned}
			\end{equation}
			where the inequality follows from
			\begin{equation*}
				\lambda_1^k(ua_k-v)\geq \lambda_1^k (3a_k-a_k)=2\lambda_k.
			\end{equation*}
			Similarly,
			\begin{equation}
				c(x_1, B_k, D_k,\dots,E_k)-c(x_1, A_k, D_k,\dots,E_k)=2\lambda_ks_k^2\left(\lambda_1^k(ua_k+v)-\lambda_k(1-a_k^2)\right)>0,
			\end{equation}
			and
			\begin{equation}
					c(x_1, B_k, C_k,\dots,E_k)-c(x_1, A_k, D_k,\dots,E_k)=4\lambda_1^k\lambda_kua_ks_k^2>0.
			\end{equation}
		Thus, $(A_k,D_k,E_k,\dots,E_k)$ is the unique pointwise minimizer on $V_{k,+}$. By symmetry, $(B_k,C_k,E_k,\dots,E_k)$ is the unique pointwise minimizer on $V_{k,-}$.

		\textbf{Step 4:} Let $l\neq k$ and $x_1\in Z_l$. By the same separation estimate as in Step 2,
		in the pointwise comparison on $Z_l$, it suffices to restrict to admissible tuples contained in $Y_l^{N-1}$. For this cell, the only atoms of the target marginals \(\mu_2^k,\dots,\mu_N^k\) in \(Y_l\) are all equal to \(E_l\). Thus the only admissible tuple in this cell is \((E_l,E_l,\dots,E_l)\), and it is the unique pointwise minimizer on \(Z_l\).

		Combining this with the preceding results, the maps given in Step 1 are pointwise optimal among all admissible target tuples for $\rho$-a.e. $x_1\in S$. Therefore, these maps are the multi-marginal optimal transport maps associated with $(\rho, \mu_2^k,\dots, \mu_N^k)$.

		 \textbf{Step 5:} We show that
		 \begin{equation*}
		 	W_1(\mu_2^k, \mu_3^k)=\frac{16a_k^2s_k^{n+1}}{\mathcal{L}^n(S)}.
		 \end{equation*}

		 On the one hand, by matching $A_k$ to $C_k$ and $B_k$ to $D_k$, we have
		 \begin{equation}
		 	W_1(\mu_2^k, \mu_3^k)\leq v_k|A_k-C_k|+v_k|B_k-D_k|=\frac{16a_k^2s_k^{n+1}}{\mathcal{L}^n(S)}.
		 \end{equation}

		 On the other hand, define a function $f$ by
		 \begin{equation*}
		 	f(A_k)=f(B_k):=a_ks_k,\quad f(C_k)=f(D_k):=-a_ks_k,\quad f(E_l):=0,\quad l\in\N.
		 \end{equation*}
		 Together with \eqref{lk}, this definition shows that $f$ is $1$-Lipschitz on $\{A_k,B_k,C_k,D_k\}\cup \{E_l: l\in\N\}$. By McShane's lemma, $f$ can be extended to $\R^n$ with $\Lip(f)\leq 1$. Thus, by Kantorovich--Rubinstein duality, we have
		 \begin{equation}
		 	W_1(\mu_2^k, \mu_3^k)\geq \int f\,\d (\mu_2^k-\mu_3^k)=\frac{16a_k^2s_k^{n+1}}{\mathcal{L}^n(S)}.
		 \end{equation}

		 Therefore,
		 \begin{equation}
		 	W_1(\mu_2^k, \mu_3^k)=\frac{16a_k^2s_k^{n+1}}{\mathcal{L}^n(S)}, \quad \|T_2^k-T_3^k\|_{L^2(\rho)}=
		 	\frac{4\sqrt{2}a_k^{1/2}s_k^{(n+2)/2}}{(\mathcal{L}^n(S))^{1/2}}.
		 \end{equation}
Since $a_k=\lambda_k/\lambda_1^k\asymp\lambda_k$, the last two identities give \eqref{eq:weight-scales}. The construction fixes $S,Y,\rho$ once $n,N$ are chosen.
\end{proof}

\def\cprime{$'$}
\providecommand{\bysame}{\leavevmode\hbox to3em{\hrulefill}\thinspace}
\providecommand{\MR}{\relax\ifhmode\unskip\space\fi MR }
\providecommand{\MRhref}[2]{%
  \href{http://www.ams.org/mathscinet-getitem?mr=#1}{#2}
}
\providecommand{\href}[2]{#2}

		\bigskip

\noindent\textbf{Acknowledgements.} The authors thank the
\href{https://llm.ustc.edu.cn}{Large Language Model Public Service Platform}
at the University of Science and Technology of China for providing free usage credits.

\medskip

\noindent\textbf{Declarations.} The authors declare that they have no conflicts of interest. No data are associated with this manuscript. The central mathematical ideas of this work were conceived by the authors. Codex and ChatGPT were used to improve the language and presentation of the first arXiv version. In preparing the second version, the authors restructured the manuscript following discussions with ChatGPT and DeepSeek about its organization and mathematical arguments. Codex was also used to polish the language and presentation of the present submission. The authors take full responsibility for all mathematical arguments, references, and conclusions.

\medskip

\noindent \textbf{Funding.} This work was supported in part by the National Key R\&D Program of China (2021YFA1000900, 2021YFA1002200), the National Natural Science Foundation of China (12201596), the Shandong Provincial Natural Science Foundation (ZR2025QB05) and the Taishan Scholars Program of Shandong Province (tsqn202408059).

\end{document}